\documentclass[a4paper,twoside,11pt]{amsart}
\usepackage{amsmath}
\usepackage{amssymb}
\usepackage{amsfonts}
\usepackage{mathrsfs}
\usepackage{hyperref}
\usepackage{mathtools}
\usepackage[hyphenbreaks]{breakurl} 
\usepackage{xcolor}

\usepackage{enumerate}
\renewcommand{\theenumi}{{\upshape{(\roman{enumi})}}}

\renewcommand{\labelenumi}{\theenumi}

\usepackage[matrix,arrow,cmtip,all]{xy} 
        \theoremstyle{plain}
        \newtheorem{theorem}{Theorem}[section]
        \newtheorem{corollary}[theorem]{Corollary}
        \newtheorem{lemma}[theorem]{Lemma}
        \newtheorem{proposition}[theorem]{Proposition}
        \newtheorem{question}[theorem]{Question}

        \theoremstyle{definition}
        \newtheorem{definition}[theorem]{Definition}
        \newtheorem{example}[theorem]{Example}
        \newtheorem{notation}[theorem]{Notation}
 
        \theoremstyle{remark}
        \newtheorem{remark}[theorem]{Remark}

\usepackage{macros}
\newcommand{\cC}{\mathbf{C}} 
\newcommand{\cD}{\mathbf{D}}
\newcommand{\cK}{\mathbf{K}}
\newcommand{\cQ}{\mathbf{Q}}
\newcommand{\Coh}{\mathbf{Coh}} 
\newcommand{\Cat}{\mathbf{Cat}} 
\newcommand{\Mod}{\mathbf{Mod}}
\newcommand{\op}{\mathrm{op}} 

\newcommand{\QC}{\mathrm{QC}} 
\newcommand{\BGL}{\mathrm{BGL}} 

\newcommand{\sQ}{\mathcal{Q}} 

\newcommand{\Ass}{\mathrm{Ass}}

\DeclareMathOperator{\coker}{coker}

\DeclareMathOperator{\Isom}{Isom}
\DeclareMathOperator{\Aut}{Aut}
\let\Im\undefined
\DeclareMathOperator{\Im}{Im}

\newcommand{\spref}[1]{\href{http://stacks.math.columbia.edu/tag/#1}{#1}}

\newcommand{\Aff}{\mathbb{A}} 

\newcommand{\AlgSt}{\mathsf{AlgSt}} 

\DeclareMathOperator{\Ext}{Ext}
\newcommand{\ant}{\mathrm{ant}} 

\renewcommand{\QCoh}{\mathsf{QCoh}}
\renewcommand{\Mod}{\mathsf{Mod}}
\renewcommand{\Coh}{\mathsf{Coh}}
\newcommand{\Alg}{\mathsf{CAlg}}
\newcommand{\CMon}{\Alg}
\newcommand{\sSpec}{\mathcal{S}\mathrm{pec}}
\newcommand{\Proj}{\mathrm{Proj}}
\newcommand{\GATC}{\mathbf{GTC}}

\newcommand{\LDERF}{\mathsf{L}}

\newcommand{\weilr}{\mathbf{R}} 

\newcommand{\TF}[1]{\omega_{{#1}}}
\newcommand{\TFiso}[1]{\TF{{#1},\simeq}}
\newcommand{\ft}{\mathrm{ft}} 

\numberwithin{equation}{section}

\begin{document}

\title[Tannaka duality]{Coherent Tannaka duality and algebraicity of Hom-stacks}

\author{Jack Hall}
\address{Mathematical Sciences Institute, The Australian National
  University, Acton ACT 2601, Australia}
\email{jack.hall@anu.edu.au}
\author{David Rydh}
\address{KTH Royal Institute of Technology\\Department of Mathematics\\100 44 Stockholm\\Sweden}
\email{dary@math.kth.se}
\thanks{The first author is supported by the Australian Research Council DE150101799}
\thanks{This collaboration was supported by the G\"oran Gustafsson foundation.
The second author is also supported by the Swedish Research Council 2011-5599
and 2015-05554.}
\date{June 29, 2016}
\subjclass[2010]{Primary 14A20; Secondary 14D23, 18D10}
\keywords{Tannaka duality, algebraic stacks, Hom-stacks, Mayer--Vietoris squares, formal gluings}

\begin{abstract}
We establish Tannaka duality for noetherian algebraic stacks with affine stabilizer groups. Our main application is the existence of $\underline{\mathrm{Hom}}$-stacks in great generality.

\end{abstract}

\maketitle


\section{Introduction}
Classically, Tannaka duality reconstructs a group from its category of finite-dimensional 
representations \cite{tannaka_duality}. Various incarnations of Tannaka duality have been studied for 
decades. The focus of this article is a recent formulation for algebraic stacks \cite{lurie_tannaka} which we now recall.

Let $X$ be a noetherian algebraic stack. We denote its abelian category of coherent 
sheaves by $\Coh(X)$. If $f\colon T\to X$ is a morphism of noetherian algebraic stacks, then there is an induced pullback functor
\[
f^* \colon \Coh(X) \to \Coh(T).
\]
It is well-known that $f^*$ has the following three properties:
\begin{enumerate}
\item $f^*$ sends $\sO_X$ to $\sO_T$,
\item $f^*$ preserves the tensor product of coherent sheaves, and
\item $f^*$ is a right exact functor of abelian categories.
\end{enumerate}
Hence, there is a functor
\begin{align*}
  \Hom(T,X) &\to \Hom_{r\otimes,\simeq}\bigl(\Coh(X),\Coh(T)\bigr),\\
  (f\colon T \to X) &\mapsto \bigl(f^* \colon \Coh(X) \to \Coh(T)\bigr),
\end{align*}
where the right hand side denotes the category with objects the functors $F \colon \Coh(X) \to \Coh(T)$ satisfying conditions (i)--(iii) above and morphisms given by natural isomorphisms of functors. 

If $X$ has affine diagonal (e.g., $X$ is the quotient of a variety by an affine algebraic 
group), then the functor above is known \cite{lurie_tannaka} to be fully faithful with image consisting of \emph{tame} functors. Even though tameness of a functor is a difficult condition to verify, Lurie was able to establish some striking applications to algebraization problems. 

Various stacks of singular curves~\cite[\S4.1]{alper_kresch_versal} and log stacks can fail to have affine, quasi-affine, or even separated diagonals. In particular, for applications in moduli theory, the results of \cite{lurie_tannaka} are insufficient. The main result of this article is the following theorem, which besides removing Lurie's hypothesis of affine diagonal, obviates tameness.
\begin{theorem}\label{T:main_tannaka:aff_stab}
  Let $X$ be a noetherian algebraic stack with affine stabilizers.
  For every locally excellent algebraic stack $T$, the functor
  \[
  \Hom(T,X) \to \Hom_{r\otimes,\simeq}\bigl(\Coh(X),\Coh(T)\bigr)
  \]
  is an equivalence.
\end{theorem}
That $X$ has \emph{affine stabilizers} means that $\Aut(x)$ is affine for every field
$k$ and point $x\colon \Spec k\to X$; equivalently, the diagonal of $X$ has
affine fibers. An algebraic stack is \emph{locally excellent} if there exists a smooth presentation by an excellent scheme (see
Remark~\ref{R:excellent}); this includes every algebraic stack that is
locally of finite type over a field, over $\Z$, or over a complete local
noetherian ring. We also wish to emphasize that we do not assume that the
diagonal of $X$ is separated in Theorem~\ref{T:main_tannaka:aff_stab}. The restriction to stacks with affine stabilizers is a necessary
condition for the equivalence in Theorem~\ref{T:main_tannaka:aff_stab} (see
Theorem~\ref{T:failure_tensoriality}).

Theorem \ref{T:main_tannaka:aff_stab} is a consequence of Theorem \ref{T:main_tannaka_non-noeth}, which also gives various refinements in the non-noetherian situation and when $X$ has quasi-affine or quasi-finite diagonal.

\subsection*{Main applications}
In work with J. Alper \cite{AHR_lunafield}, Theorem \ref{T:main_tannaka:aff_stab} is applied to resolve Alper's conjecture on the local quotient structure of algebraic stacks \cite{Alper20101576}. A more immediate  application of Theorem \ref{T:main_tannaka:aff_stab} is the
following algebraicity result for $\Hom$-stacks, generalizing all
previously known results and answering \cite[Question
1.4]{MR2786662}.
\begin{theorem}\label{T:main_alg_hom}
  Let $Z\to S$ and $X\to S$ be morphisms of algebraic stacks such that
  $Z \to S$ is proper and flat of finite presentation,
  and $X \to S$ is locally of
  finite presentation, quasi-separated, and has affine stabilizers. Then
  \begin{enumerate}
  \item \label{TI:main_alg_hom:rep} the stack $\underline{\Hom}_S(Z,X)\colon
    T\mapsto \Hom_S(Z\times_S T,X)$, is algebraic;
  \item \label{TI:main_alg_hom:qs} the morphism $\underline{\Hom}_S(Z,X)\to S$ is locally of finite
    presentation, quasi-separated, and has affine stabilizers; and
  \item \label{T:main_alg_hom:diag} if $X\to S$ has affine (resp.\ quasi-affine,
    resp.\ separated) diagonal, then so has $\underline{\Hom}_S(Z,X)\to S$.
  \end{enumerate}
\end{theorem}
Theorem \ref{T:main_alg_hom} has already seen applications to log geometry \cite{wise_logmaps}, an area which provides a continual source of stacks that are neither Deligne--Mumford nor have separated diagonals. In general, the condition that
$X$ has affine stabilizers is necessary (see Theorem \ref{T:failure_hom-stack}).

There are analogous algebraicity results for Weil restrictions (that is, restrictions of scalars).
\begin{theorem}\label{T:main_alg_weilr}
  Let $f\colon Z \to S$ and $g\colon X\to Z$ be morphisms of algebraic stacks
  such that $f$ is proper and flat of finite presentation and
  $f\circ g$ is locally of finite presentation, quasi-separated and has affine
  stabilizers. Then
  \begin{enumerate}
  \item the stack $f_*X=\weilr_{Z/S}(X)\colon T\mapsto
    \Hom_Z(Z\times_S T,X)$ is algebraic;
  \item the morphism $\weilr_{Z/S}(X)\to S$ is locally of finite
    presentation, quasi-separated and has affine stabilizers; and
  \item if $g$ has affine (resp.\ quasi-affine, resp.\ separated) diagonal,
    then so has $f_*X\to S$.
  \end{enumerate}
\end{theorem}
When $Z$ has finite diagonal and $X$ has quasi-finite and
separated diagonal,
Theorems~\ref{T:main_alg_hom} and~\ref{T:main_alg_weilr} were proved in
\cite[Thms.~3 \& 4]{MR3148551}. 
In Corollary~\ref{C:main_alg_loc-approx}, we also 
excise the finite presentation assumptions on $X \to S$ in Theorems \ref{T:main_alg_hom} and \ref{T:main_alg_weilr}, generalizing the results of \cite[Thm.~2.3 \& Cor.~2.4]{hallj_dary_g_hilb_quot} for stacks with quasi-finite diagonal. 

\subsection*{Application to descent}
If $X$ has quasi-affine diagonal, then it is well-known that it is a stack for the fpqc topology \cite[Cor.~10.7]{MR1771927}. In general, it is only known that algebraic stacks satisfy effective descent for fppf coverings. Nonetheless,
using that $\QCoh$ is a stack for the fpqc-topology and Tannaka duality,
we are able to establish the 
following result.
\begin{corollary}\label{C:application-fpqc-stack}
  Let $X$ be a quasi-separated algebraic stack with affine
  stabilizers. Let $\pi\colon T'\to T$ be an fpqc covering such
  that $T$ is a locally excellent stack and $T'$ is locally noetherian.
  Then $X$ satisfies effective
  descent for $\pi$.
\end{corollary}

\subsection*{Application to completions}
Another application concerns completions.
\begin{corollary}\label{C:application-completions}
Let $A$ be a noetherian ring and let $I\subseteq A$ be an ideal. Assume that $A$ is
complete with respect to the $I$-adic topology. Let $X$ be a noetherian algebraic stack
and consider the natural morphism
\[
X(A)\to \varprojlim X(A/I^n)
\]
of groupoids. This morphism is an equivalence if either
\begin{enumerate}
\item $X$ has affine stabilizers
  and $A$ is excellent (or merely a G-ring); or
\item\label{ci:completions:qg}
  $X$ has quasi-affine diagonal; or
\item\label{ci:completions:qf}
  $X$ has quasi-finite diagonal.
\end{enumerate}
\end{corollary}
Using derived methods,
Corollary~\ref{C:application-completions}\ref{ci:completions:qg} was
recently proved for non-noetherian complete
rings $A$ \cite{2015arXiv150701925B,2014arXiv1404.7483B}. That $X$ has affine stabilizers
in Corollary~\ref{C:application-completions} is necessary (see Theorem \ref{T:failure_completions}).

\subsection*{On the proof of Tannaka duality}
We will discuss the proof of Theorem \ref{T:main_tannaka_non-noeth}, the refinement of 
Theorem \ref{T:main_tannaka:aff_stab}. The reason for this is that it is much more 
convenient from a technical standpoint to consider the problem in the setting of 
quasi-coherent sheaves on potentially non-noetherian algebraic stacks. 

So let $T$ and $X$ be algebraic stacks and let $\QCoh(T)$ and $\QCoh(X)$ denote their respective abelian categories of quasi-coherent sheaves. We will assume that $X$ is quasi-compact and quasi-separated. Our principal concern is the properties of the functor
\begin{align*}
\omega_X(T)\colon  \Hom(T,X) &\to \Hom_{c\otimes}(\QCoh(X),\QCoh(T)),\\
  (f\colon T \to X) &\mapsto (f^* \colon \QCoh(X) \to \QCoh(T)),
\end{align*}
where the right hand side denotes the additive functors
$F\colon \QCoh(X) \to \QCoh(T)$ satisfying
\begin{enumerate}
\item $F(\sO_X)=\sO_T$,
\item $F$ preserves the tensor product, and
\item $F$ is right exact and preserve (small) direct sums.
\end{enumerate}
We call such $F$ \emph{cocontinuous tensor functors}.

An algebraic stack $X$ has the \emph{resolution property} if every quasi-coherent sheaf is 
a quotient of a direct sum of vector bundles. In 
Theorem~\ref{T:resprop-tensorial} we establish the equivalence of $\TF{X}(T)$ when $X$ has affine diagonal and the resolution
property. This result has appeared in various forms in the work of others (cf.\ Sch\"appi
\cite[Thm.~1.3.2]{2012arXiv1206.2764S}, Savin~\cite{MR2207852} and Brandenburg \cite[Cor.~5.7.12]{brandenburg_thesis}) and forms an essential stepping stone in the proof of our main theorem (Theorem \ref{T:main_tannaka_non-noeth}).

In general, there are stacks---even schemes---that do not have the resolution property. Indeed, if $X$ has the resolution property, then $X$ has at least affine diagonal \cite[Prop.~1.3]{MR2108211}. Our proof uses the following
three ideas to overcome this problem:
\begin{enumerate}
\item If $U\subseteq X$ is a quasi-compact open immersion and $\QCoh(X)\to
  \QCoh(T)$ is a tensor functor, then there is an induced tensor functor
  $\QCoh(U)\to \QCoh(V)$ where $V\subseteq T$ is the ``inverse image of $U$''.
  The proof of this is based on ideas of
  Brandenburg and Chirvasitu~\cite{MR3144607}. (Section~\ref{S:tensor-localization})
\item If $X$ is an infinitesimal neighborhood of a stack with the
  resolution property, then $\TF{X}(T)$ is an equivalence
  for all $T$. (Section~\ref{S:main-lemma})
\item There is a constructible stratification of $X$ into
  stacks with affine diagonal and the resolution property (Proposition~\ref{P:filtrations-exists}).
  We deduce the main
  theorem by induction on the number of strata using formal
  gluings~\cite{MR1432058,mayer-vietoris}. This step uses special cases of
  Corollaries~\ref{C:application-fpqc-stack} and
  \ref{C:application-completions}. (Sections~\ref{S:formal-gluings} and~\ref{S:filtrations})
\end{enumerate}
In the third step, we assume that our functors preserve sheaves of finite type.
\subsection*{Open questions}
Concerning (ii), it should be noted that we do not know the answers to the
following two questions.
\begin{question}\label{Q:deformation-of-resprop}
If $X_0$ has the resolution property and $X_0\inj X$ is a nilpotent closed
immersion, then does $X$ have the resolution property?
\end{question}
The question has an affirmative answer if $X_0$ is cohomologically affine,
e.g., $X_0=B_kG$ where $G$ is a linearly
reductive group scheme over $k$. The question is open if
$X_0=B_kG$ where $G$ is not linearly reductive, even
if $X=B_{k[\epsilon]}G_\epsilon$ where $G_\epsilon$ is a deformation of $G$
over the dual numbers~\cite{mathoverflow_groups-over-dual-numbers}.
\begin{question}
If $X_0\inj X$ is a nilpotent closed immersion and $\TF{X_0}(T)$ is
an equivalence, is then $\TF{X}(T)$ an equivalence?
\end{question}
Step (ii) answers neither of these questions but uses a special case of
the first question (Lemma~\ref{L:res_strange}) and
the conclusion (Main Lemma~\ref{L:funny_nil_res}) is a special case
of the second question.

The following technical question also arose in this research.
\begin{question}\label{Q:pseudo-geometric}
  Let $X$ be an algebraic stack with quasi-compact and quasi-separated
  diagonal and affine stabilizers. Let $k$ be
  a field. Is every morphism $\Spec k \to X$ affine?
\end{question}
If $X$ \'etale-locally has quasi-affine diagonal, then Question
\ref{Q:pseudo-geometric} has an affirmative answer (Lemma
\ref{L:fields-to-stacks}). This makes finding counterexamples
extraordinarily difficult and thus very interesting. This question
arose because if $\Spec k \to X$ is non-affine, then $\omega_X(\Spec
k)$ is not fully faithful (Theorem~\ref{T:failure_full}).
This explains our restriction to natural isomorphisms in Theorem
\ref{T:main_tannaka:aff_stab}. Note that every morphism $\Spec k \to X$
as in Question~\ref{Q:pseudo-geometric} is at least
quasi-affine~\cite[Thm.~B.2]{MR2774654}. We do not know
the answer to the question even if $X$ has separated diagonal and is of
finite type over a field.
\subsection*{On the applications}
Let $T$ be a noetherian and locally excellent algebraic stack and let $Z$ be a closed substack defined by a coherent ideal $J\subseteq \sO_T$.
Let $Z^{[n]}$ be the closed substack defined by $J^{n+1}$. 
Assume that the natural functor $\Coh(T)\to \varprojlim_n \Coh(Z^{[n]})$ is an
equivalence of categories. Then an immediate consequence of Tannaka
duality (Theorem~\ref{T:main_tannaka:aff_stab}) is that
\[
\Hom(T,X)\to \varprojlim \Hom(Z^{[n]},X)
\]
is an equivalence of categories for every noetherian algebraic stack $X$
with affine stabilizers.
This applies in particular
if $A$ is excellent and $I$-adically complete and
$T=\Spec A$ and $Z=\Spec A/I$; this gives
Corollary~\ref{C:application-completions}. More generally, it also
applies if $T$ is proper over $\Spec A$ and ${Z=T\times_{\Spec A}
\Spec A/I}$ (Grothendieck's existence theorem). This latter case is fed into Artin's criterion to prove Theorem~\ref{T:main_alg_hom} (the remaining hypotheses have largely been verified elsewhere). 

There are also non-proper
stacks $T$ satisfying $\Coh(T)\to \varprojlim_n \Coh(Z^{[n]})$, such
as global quotient stacks with proper good moduli spaces (see \cite{geraschenko-zb_fGAGA,MR2092127} for some special cases). This featured in the resolution of Alper's conjecture \cite{AHR_lunafield}.

Such statements, and their derived versions, were also recently considered by
Halpern-Leistner--Preygel~\cite{2014arXiv1402.3204H}. There, they
considered variants of our Theorem \ref{T:main_alg_hom}. For their
algebraicity results, their assumption was similar to assuming
that $\Coh(T) \to \varprojlim_n \Coh(Z^{[n]})$ was an equivalence
(though they also considered other derived versions), and
that $X \to S$ was locally of finite presentation with affine
diagonal. 
\subsection*{Relation to other work}
As mentioned in the beginning of the Introduction, Lurie identifies the
image of $\TF{X}(T)$ with the \emph{tame} functors when $X$ is quasi-compact with affine diagonal \cite{lurie_tannaka}. Tameness means that
faithful flatness of objects is preserved. This is a very strong assumption that
makes it possible to directly pull back a smooth presentation of $X$ to
a smooth covering of $T$ and deduce the result by descent.
Note that every tensor functor preserves coherent flat objects---these are
vector bundles and hence dualizable---but this does not imply that flatness of
quasi-coherent objects are preserved.
Lurie's methods
also work for non-noetherian $T$.

Brandenburg and Chirvasitu have
shown that $\TF{X}(T)$ is an equivalence for every
quasi-compact and quasi-separated scheme $X$~\cite{MR3144607}, also for
non-noetherian $T$. The key idea of their proof is the tensor localization
that we have adapted in Section~\ref{S:tensor-localization}. Using this
technique, we give a slightly simplified proof of their theorem in
Theorem~\ref{T:scheme-tensorial}. 

When $X$ has quasi-affine diagonal, derived variants of Theorem 
\ref{T:main_tannaka:aff_stab} have recently been considered by various authors 
\cite{MR3048606,2014arXiv1404.7483B,2015arXiv150701925B}. Specifically, they were 
concerned with symmetric monoidal $\infty$-functors $G \colon D(X) \to D(T)$ between stable $\infty$-categories 
of quasi-coherent sheaves.

It is not obvious how to go from a tensor functor $F \colon
\QCoh(X) \to \QCoh(T)$ to a symmetric monoidal $\infty$-functor $\LDERF F
\colon D(X) \to D(T)$, so our results cannot be deduced from the derived
perspective. When $T$ is locally noetherian, however, our result is stronger
than~\cite[Thm.~1.4]{2015arXiv150701925B}. Indeed, the functors $G \colon D(X) \to D(T)$
are assumed to preserve derived tensor products, connective complexes (i.e.,
are right $t$-exact) and pseudo-coherent complexes and hence induces a
right-exact tensor functor $H^0(G) \colon \QCoh(X)
\to \QCoh(T)$ preserving sheaves of finite type.
When $X$ has finite stabilizers, the right $t$-exactness is
sometimes automatic~\cite{MR3048606,2014arXiv1404.7483B,16334}.

We do not address the Tannaka \emph{recognition} problem, i.e., which symmetric
monoidal categories arise as the category of quasi-coherent sheaves on an
algebraic stack. For gerbes, this has been done in characteristic zero by
Deligne~\cite[Thm.~7.1]{deligne_categories-tannakiennes}. For stacks with the
resolution property, this has been done by
Sch\"appi~\cite[Thm.~1.4]{schappi_which_geometric},~\cite[Thms.~1.2.2,
  5.3.10]{2015arXiv150504596S}. Similar results from the derived perspective
have been considered by Wallbridge \cite{2012arXiv1204.5787W} and Iwanari
\cite{2014arXiv1409.3321I}.

\subsubsection*{Acknowledgments}
We would like to thank Dan Abramovich for asking us whether Theorem \ref{T:main_alg_hom} held, which was the original motivation for this article.
We would also like to thank Martin Brandenburg for his many comments on a
preliminary draft and for sharing his dissertation with us. In particular, he
made us aware of Proposition~\ref{P:char_aff_brandenburg}, which greatly
simplified Corollary~\ref{C:pb_affine_stack}.
Finally, we would like to thank Bhargav Bhatt for several useful comments and
suggestions, Jarod Alper for several interesting and supportive discussions,
and Andrew Kresch for some encouraging remarks.

\section{Symmetric monoidal categories}
  A \emph{symmetric monoidal category} is the data of a category
  $\cC$, a tensor product $\otimes_{\cC} \colon \cC \times \cC \to
  \cC$, and a unit $\sO_{\cC}$ that together satisfy various
  naturality, commutativity, and associativity properties
  \cite[VII.7]{MR1712872}.
  A symmetric monoidal category $\cC$ is \emph{closed} if for any
  $M\in \cC$ the functor $-\otimes_{\cC} M \colon \cC \to \cC$ admits
  a right adjoint, which we denote as $\sHom_{\cC}(M,-)$.
  \begin{example}
    Let $A$ be a ring; then the category of $A$-modules, $\Mod(A)$, together with its 
    tensor product $\otimes_A$, is a symmetric monoidal category with unit $A$. In fact, 
    $\Mod(A)$ is even closed: the right adjoint to $-\otimes_A M$ is the $A$-module 
    $\Hom_A(M,-)$. If $A$ is 
    noetherian, then the subcategory of finite $A$-modules, $\Coh(A)$, is also 
    a closed symmetric monoidal category. 
  \end{example}  
  A functor $F\colon \cC \to \cD$ between symmetric monoidal
  categories is \emph{lax symmetric monoidal} if for each $M$ and $M'$ of $\cC$
  there are natural maps $F(M) \otimes_{\cD} F(M') \to
  F(M\otimes_{\cC} M')$ and $\sO_{\cD} \to
  F(\sO_{\cC})$ that are compatible with the symmetric monoidal structure. If
  these maps are both isomorphisms, then $F$ is \emph{symmetric monoidal}. Note
  that if $F \colon \cC \to \cD$ is a symmetric monoidal functor, then a right
  adjoint $G \colon \cD \to \cC$ to $F$ is always lax symmetric monoidal.
  \begin{example}
    Let $\phi\colon A \to B$ be a ring homomorphism. The functor $-\otimes_A B \colon 
    \Mod(A) \to \Mod(B)$ is symmetric monoidal. It admits a right adjoint, $\Mod(B) \to 
    \Mod(A)$, which is given by the forgetful functor. This forgetful functor is lax monoidal, 
    but not monoidal.
  \end{example}
  If $\cC$ is a symmetric monoidal category, then a \emph{commutative
    $\cC$-algebra} consists of an object $A$ of $\cC$ together with a
  multiplication $m\colon A \otimes_{\cC} A \to A$ and a unit $e_A
  \colon \sO_{\cC} \to A$ with the expected properties
  \cite[VII.3]{MR1712872}. Let $\CMon(\cC)$ denote the category of commutative $\cC$-algebras. The category $\CMon(\cC)$ is naturally endowed with a symmetric monoidal structure that makes the forgetful functor $\CMon(\cC) \to \cC$ symmetric monoidal.  
  \begin{example}
    If $A$ is a ring, then $\CMon(\Mod(A))$ is the category of commutative $A$-algebras.
  \end{example}
  The following observation will be used frequently: if $F \colon \cC \to \cD$ is a lax symmetric monoidal
  functor and $A$ is a commutative $\cC$-algebra, then $F(A)$ is a
  commutative $\cD$-algebra.

\section{Abelian tensor categories}
An \emph{abelian tensor category} is a symmetric
monoidal category that is abelian and the tensor product is right exact and preserves finite direct sums in each variable (i.e., preserves all finite colimits in each variable). 

Recall that an abelian category is \emph{Grothendieck} if it is closed under small direct 
sums, filtered colimits are exact, and it has a generator \cite[Tag \spref{079A}]{stacks-project}. Also, recall that a functor $F \colon \cC \to \cD$ between two Grothendieck abelian categories is \emph{cocontinuous} if it is right-exact and preserves small direct sums, equivalently, it preserves all small colimits.

A \emph{Grothendieck abelian tensor category}
is an abelian tensor category such that the underlying abelian category
is Grothendieck abelian and the tensor product is cocontinuous in each variable. By the Special Adjoint Functor Theorem
\cite[Prop.~8.3.27(iii)]{MR2182076}, if $\cC$ is a Grothendieck abelian
tensor
category, then it is also closed.
\begin{example}
  Let $A$ be a ring. Then $\Mod(A)$ is a Grothendieck abelian tensor category. If $A$ is 
  noetherian, then $\Coh(A)$ is an abelian tensor category but not Grothendieck 
  abelian---it is not closed under small direct sums.
\end{example}

A \emph{tensor functor} $F\colon \cC \to \cD$ is an additive
symmetric monoidal functor between abelian tensor categories. Let
$\GATC$ be the $2$-category of Grothendieck abelian tensor categories
and cocontinuous
tensor functors. By the Special Adjoint Functor Theorem, if $F\colon
\cC \to \cD$ is a cocontinuous tensor functor, then $F$ admits a (lax symmetric
monoidal) right adjoint.
\begin{example}
  Let $T$ be a ringed site. The category of $\sO_T$-modules $\Mod(T)$
  is a Grothendieck
  abelian tensor category with unit $\sO_T$ and the internal Hom is the
  functor $\sHom_{\sO_T}(M,-)$. 
\end{example}
\begin{example}\label{E:ab_tens_cat_alg_stacks}
  Let $X$ be an algebraic stack. The category of quasi-coherent
  sheaves $\QCoh(X)$ is a Grothendieck abelian tensor category with
  unit $\sO_X$
  \cite[Tag \spref{0781}]{stacks-project}. The internal $\Hom$ is
  $\QC(\sHom_{\sO_X}(M,-))$, where $\QC$ denotes the
  quasi-coherator (the right adjoint to the inclusion of the category
  of quasi-coherent sheaves in the category of lisse-\'etale
  $\sO_X$-modules). If $X$ is an algebraic stack, then
  $\CMon(\QCoh(X))$ is the symmetric monoidal category of
  quasi-coherent $\sO_{X}$-algebras.

  If $\map{f}{X}{Y}$ is a morphism of algebraic stacks, then the
  resulting functor $\map{f^*}{\QCoh(Y)}{\QCoh(X)}$ is a cocontinuous
  tensor functor. If $f$ is flat, then $f^*$ is exact. We always denote the right adjoint of $f^*$ by $\map{f_*}{\QCoh(X)}{\QCoh(Y)}$. If $f$ is quasi-compact and quasi-separated, then $f_*$ coincides with the pushforward of lisse-\'etale $\sO_X$-modules \cite[Lem.~6.5(i)]{MR2312554}.  In particular, if $f$ is
  quasi-compact and quasi-separated, then $f_* \colon \QCoh(X) \to
  \QCoh(Y)$ preserves directed colimits (work smooth-locally on $Y$ and then apply \cite[Tag~\spref{0738}]{stacks-project}) and is lax symmetric monoidal.
\end{example}
\begin{definition}
Given abelian tensor categories $\cC$ and
$\cD$, we let $\Hom_{c\otimes}(\cC,\cD)$ (resp.~$\Hom_{r\otimes}(\cC,\cD)$) denote the
category of cocontinuous (resp.~right exact) tensor functors and natural
transformations. The transformations are required to be natural with respect
to both homomorphisms and the symmetric monoidal structure.
We let $\Hom_{c\otimes,\simeq}(\cC,\cD)$ (resp.~$\Hom_{r\otimes,\simeq}(\cC,\cD)$) denote the groupoid of cocontinuous (resp.~right exact) tensor functors and
natural isomorphisms.
\end{definition}

We conclude this section with some useful facts for the paper. We first consider modules over algebras,
which are addressed, for example, in Brandenburg's
thesis~\cite[\S5.3]{brandenburg_thesis} in even greater generality.

\subsection{Modules over an algebra in tensor categories}\label{SS:modules}
Let $\cC$ be a Grothen\-dieck
abelian tensor category and let $A$ be a commutative
$\cC$-algebra. Define $\Mod_{\cC}(A)$ to be the category of
$A$-modules. Objects are pairs $(M,a)$, where $M\in \cC$
and $a \colon A\otimes_{{\cC}} M \to M$ is an action of $A$ on $M$.
Morphisms $\phi \colon (M,a) \to (M',a')$ in $\Mod_{\cC}(A)$ are those
$\cC$-morphisms $\phi \colon M \to M'$ that preserve the respective
actions. We identify $A$ with $(A,m)\in \Mod_{\cC}(A)$ where $m\colon
A \otimes_{\cC} A \to A$ is the multiplication. It is straightforward
to show that $\Mod_{\cC}(A)$ is a Grothendieck abelian tensor
category, with tensor product $\otimes_A$ and unit $A$, and the
natural forgetful functor $\Mod_{\cC}(A) \to \cC$ preserves all limits
and colimits \cite[\S4.3]{MR2182076}.

  If $s \colon A \to B$ is a $\cC$-algebra homomorphism, 
  then there is a natural cocontinuous tensor functor
  \[
  s^* \colon \Mod_{\cC}(A) \to \Mod_{\cC}(B),\quad
  (M,a) \mapsto (B\otimes_A M,B\otimes_A a). 
  \]
  Suppose $f^* \colon \cC \to \cD$ is a cocontinuous tensor functor with
  right adjoint $f_* \colon \cD \to \cC$. If
  $A$ is a commutative $\cC$-algebra, then there is a natural induced
  cocontinuous tensor functor
  \[
  f^*_A \colon \Mod_{\cC}(A) \to \Mod_{\cD}(f^*A),\quad
  (M,a) \mapsto (f^*M,f^*a).
  \]
  Noting that $\epsilon\colon f^*f_*\sO_{\cD} \to \sO_{\cD}$ is a $\cD$-algebra homomorphism, there is a natural induced cocontinuous tensor functor
  \[
  \bar{f}^* \colon \Mod_{\cC}(f_*\sO_{\cD}) \xrightarrow{f^*_{f_*\sO_{\cD}}}
  \Mod_{\cD}(f^*f_*\sO_{\cD})\xrightarrow{\epsilon^*} \Mod_{\cD}(\sO_{\cD}) = \cD.
  \]
  Moreover, if we let $\eta\colon \sO_{\cC}\to
  f_*f^*\sO_{\cC}=f_*\sO_{\cD}$ denote the unit, then $f^*=\bar{f}^*
  \eta^*$. We have the following striking
  characterization of module categories \cite[Prop.~5.3.1]{brandenburg_thesis}.
  \begin{proposition}\label{P:char_aff_brandenburg}
    Let $\cC$ be a Grothendieck abelian tensor category and let $A$ be
    a commutative algebra in $\cC$. Then for every Grothendieck
    abelian tensor category $\cD$, there is an equivalence of
    categories
    \begin{align*}
    \Hom_{c\otimes}(\Mod_{\cC}(A),\cD) \simeq \{(F,h) \colon &F \in
    \Hom_{c\otimes}(\cC,\cD), \\
    &h \in
    \Hom_{\CMon(\cD)}(F(A),\sO_{\cD})\},
    \end{align*}
    where a morphism $(F,h)\to (F',h')$ is a natural transformation
    $\alpha\colon F\to F'$ such that $h=h'\circ \alpha(A)$.
  \end{proposition}
The following corollary is immediate (see
\cite[Cor.~5.3.7]{brandenburg_thesis}).
  \begin{corollary}\label{C:pb_affine_stack}
    Let $p\colon Y' \to Y$ be an affine morphism of algebraic
    stacks. Let $X$ be an algebraic stack and let $g^*\colon\QCoh(Y)
    \to \QCoh(X)$ be a cocontinuous tensor functor. If $X'$ is the
    affine $X$-scheme $\sSpec_{X}(g^*p_*\sO_{Y'})$
    with structure morphism $p'\colon X' \to X$, then there is a
    $2$-cocartesian diagram in $\GATC$:
    \[
    \xymatrix{\QCoh(X') & \QCoh(Y') \ar[l]_{g'^*} \\
      \QCoh(X) \ar[u]^{p'^*} & \QCoh(Y)\ar[u]_{p^*} \ar[l]_{g^*}.}
    \]
    Moreover, the natural transformation $g^*p_*\Rightarrow p'_*g'^*$ is an
    isomorphism.
  \end{corollary}
  Note that if $g^*$ comes from a morphism $g\colon X \to Y$,
  then $X'\cong X\times_Y Y'$.

\subsection{Inverse limits of abelian tensor categories}
  We will now briefly discuss some inverse limits that will be crucial
  when we apply Tannaka duality to establish the algebraicity of Hom-stacks in
  Theorem~\ref{T:main_alg_hom}. The following notation will be useful.  
  \begin{notation}
    Let $i\colon Z \to X$ be a closed immersion of algebraic stacks defined by a
    quasi-coherent ideal $I$. For each integer $n\geq 0$, we let $i^{[n]} \colon Z^{[n]} \to 
    X$ denote the closed immersion defined by the quasi-coherent ideal $I^{n+1}$, which 
    we call the \emph{$n$th infinitesimal neighborhood of $Z$.}
  \end{notation}
  Let $X$ be a noetherian algebraic stack and let $i\colon Z \to X$ be a closed immersion. 
  Let $\Coh(X,Z)$ denote the category $\varprojlim_n \Coh(Z^{[n]})$. The arguments of \cite[Tag~\spref{087X}]{stacks-project} easily extend to establish the following:
  \begin{enumerate}
  \item $\Coh(X,Z)$ is an abelian tensor category with
    unit $\{\sO_{Z^{[n]}}\}$ and tensor product
    $\{M_n\}_{n\geq 0}\otimes \{N_n\}_{n\geq 0}
    =\{M_n \otimes_{\sO_{Z^{[n]}}} N_n\}_{n\geq 0}$;
  \item if $p\colon U \to X$ is smooth and quasi-compact, then the restriction $\Coh(X,Z) \to
    \Coh(U,U\times_X Z)$ is an exact tensor functor; and
  \item exactness in $\Coh(X,Z)$ may be checked on a smooth covering of $X$. 
  \end{enumerate}
  If $\{f_n \colon M_n \to N_n\}_{n\geq 0}$ is a morphism in $\Coh(X,Z)$, then it is easily determined that $\coker( \{f_n\}_{n\geq 0} ) \cong \{\coker f_n\}_{n\geq 0}$. Computing $\ker(\{f_n\}_{n\geq 0})$ is more involved. We will need the following lemma.
\begin{lemma}\label{L:re_coh}
  $\Coh(X,Z)$ is the limit of the inverse system of categories $\{\Coh(Z^{[n]})\}_{n\geq 0}$ in the $2$-category of abelian tensor categories with right exact tensor functors and natural isomorphisms of tensor functors. 
\end{lemma}
\begin{proof}
  It remains to verify that if $F_n \colon \cC \to \Coh(Z^{[n]})$ is a compatible sequence of 
right exact abelian tensor functors, then the induced abelian tensor functor $\varprojlim_n 
F_n \colon \cC \to \Coh(X,Z)$ is right exact. The explicit description of cokernels in 
$\Coh(X,Z)$ shows that this is the case. 
\end{proof}

\section{Tensorial algebraic stacks}\label{S:tensorial}
Let $T$ and $X$ be algebraic stacks. There is an induced
functor
\[
\map{\TF{X}(T)}{\Hom(T,X)}{\Hom_{c\otimes}\bigl(\QCoh(X),\QCoh(T)\bigr)}
\]
that takes a morphism $f$ to $f^*$. We also let  $\Hom_{c\otimes}^\ft\bigl(\QCoh(X),\QCoh(T)\bigr)$ denote the
full subcategory of functors that preserve sheaves of finite type. Similarly, we let 
$\Hom_{c\otimes,\simeq}\bigl(\QCoh(X),\QCoh(T)\bigr)$ denote the subcategory of
natural isomorphisms of functors. Clearly, $\omega_X(T)$ factors through all of these subcategories and 
we let $\TFiso{X}(T)$, $\TF{X}^\ft(T)$ and $\TFiso{X}^\ft(T)$ denote the respective 
factorizations. Note that when $X$ and $T$ are locally noetherian, the natural functor:
\[
\Hom_{r\otimes}\bigl(\Coh(X),\Coh(T)\bigr)\to \Hom_{c\otimes}^\ft\bigl(\QCoh(X),\QCoh(T)\bigr)
\]
is an equivalence of categories. Thus, Theorem~\ref{T:main_tannaka:aff_stab}
says that $\TFiso{X}^\ft(T)$ is an equivalence.

Since $\QCoh(-)$ is a stack in the fpqc topology,
the target categories of the functors $\TF{X}$, $\TFiso{X}$, $\TF{X}^{\ft}$ and $\TFiso{X}^{\ft}$ are stacks in the fpqc topology when
varying $T$---for an elaborate proof of this,
see~\cite[Thm.~1.1]{2012arXiv1206.0076L}. The source categories $\Hom(T,X)$
are groupoids and, when varying $T$, form a stack for the fppf topology in
general and for
the fpqc topology when $X$ has quasi-affine
diagonal~\cite[Cor.~10.7]{MR1771927}.
\begin{definition}
  Let $T$ and $X$ be algebraic stacks. We say that a tensor functor $f^* \colon
  \QCoh(X) \to \QCoh(T)$ is
  \emph{algebraic} if it arises from a morphism of algebraic
  stacks $f\colon T \to X$. If $f,g\colon T\to X$ are morphisms, then
  a natural transformation $\tau\colon f^*\Rightarrow g^*$ of tensor functors
  is \emph{realizable} if it is induced by a $2$-morphism $f\Rightarrow g$.
  We say that $X$ is \emph{tensorial} if
  $\TF{X}(T)$ is an equivalence for every algebraic stack
  $T$, or equivalently, for every affine
  scheme $T$~\cite[Def.~3.4.4]{brandenburg_thesis}.
\end{definition}
We begin with a descent lemma.
\begin{lemma}\label{L:fppf_local_tannakian}
  Let $X$ be an algebraic stack. Let $p\colon T' \to T$ be a morphism 
  of algebraic stacks that is covering for the fpqc topology.
  Let $T'' = T'\times_T T'$ and 
  $T'''=T'\times_T T'\times_T T'$. Assume that $p$ is a morphism of effective
  descent for $X$ (e.g., $p$ is flat and locally of
  finite presentation). 
  \begin{enumerate}
  \item \label{LI:fppf_local_tannakian:faith}
    Let $f_1,f_2\colon T\to X$ be morphisms and let $\tau,\tau'\colon
    f_1\Rightarrow f_2$ be $2$-morphisms. If $p^*\tau=p^*\tau'\colon f_1\circ
    p\Rightarrow f_2\circ p$ then $\tau=\tau'$.
  \item \label{LI:fppf_local_tannakian:ff}
    Let $f_1,f_2\colon T\to X$ be morphisms and let $\gamma\colon
    f_1^*\Rightarrow f_2^*$ be a natural transformation. If $p^*\gamma\colon
    p^*f_1^*\Rightarrow p^*f_2^*$ is realizable and $\TF{X}(T'')$ is faithful,
    then $\gamma$ is realizable.
  \item \label{LI:fppf_local_tannakian:alg} Let $f^* \colon \QCoh(X) \to \QCoh(T)$ be a cocontinuous tensor functor. If 
    $p^*f^*$ is algebraic, $\TFiso{X}(T'')$ is fully faithful and $\TF{X}(T''')$ is faithful, then 
    $f^*$ is algebraic.
  \end{enumerate}
  Let $\omega\in\{\TF{X},\TFiso{X},\TF{X}^\ft,\TFiso{X}^\ft\}$
  and $P\in \{\text{all}, \text{locally
    noetherian}, \text{locally excellent}\}$. If $\omega(T)$ is faithful
  (resp.\ fully faithful, resp.\ an equivalence) for every affine scheme $T$
  with property $P$, then $\omega(T)$ is faithful (resp.\ fully faithful,
  resp.\ an equivalence) for every algebraic stack $T$ with property $P$.
\end{lemma}
\begin{proof}
  It is sufficient to observe that $\Hom(-,X)$ is a stack in groupoids for the covering $p$
  and $\Hom_{c\otimes}(\QCoh(X),\QCoh(-))$
  is an fpqc stack in categories, so the result boils down to a 
  straightforward and general result for a $1$-morphism of such stacks.
\end{proof}
We next recall two basic lemmas on tensorial stacks. The
first is the combination of~\cite[Cor.~5.3.4 \&~5.6.4]{brandenburg_thesis}.
\begin{lemma}\label{L:qaff-tensorial:rel}
  Let $q\colon X' \to X$ be a quasi-affine morphism of algebraic stacks.
  If $T$ is an algebraic stack and $\TF{X}(T)$ is faithful, fully
  faithful or an equivalence; then so is $\TF{X'}(T)$. In
  particular, if $X$ is tensorial, then so is $X'$.
\end{lemma}
\begin{proof}
  Since $q$ is the composition
  of a quasi-compact open immersion followed by an affine morphism, 
  it suffices to treat these two cases separately. When $q$ is affine the
  result is an easy consequence of
  Proposition \ref{P:char_aff_brandenburg}. If $q$ is a quasi-compact open
  immersion, then the counit $q^*q_*\to \id{\QCoh(X')}$ is an isomorphism; the
  result now follows from \cite[Prop.~2.3.6]{MR3144607}.
\end{proof}
The second lemma is well-known (e.g., it is a very special case
of~\cite[Thm.~3.4.2]{MR3144607}).
\begin{lemma}\label{L:qaff-tensorial:abs}
Every quasi-affine scheme is tensorial.
\end{lemma}
\begin{proof}
  By Lemma~\ref{L:qaff-tensorial:rel}, it is
  sufficient to prove that $X=\Spec\mathbb{Z}$ is tensorial, which
  is well-known. We refer the interested reader to
  \cite[Cor.~2.2.4]{MR3144607} or \cite[Cor.~5.2.3]{brandenburg_thesis}.
\end{proof}
Lacking an answer to Question \ref{Q:pseudo-geometric} in general, we
are forced to make the following definition to treat natural
transformations that are not isomorphisms.
\begin{definition}
  An algebraic stack $X$ is \emph{affine-pointed} if
  every morphism $\Spec k \to X$, where $k$ is a field, is affine.
\end{definition}
Note that if $X$ is affine-pointed, then it has affine stabilizers.
The following lemma shows that many algebraic stacks with affine
stabilizers that are encountered in practice are affine-pointed.
\begin{lemma}\label{L:fields-to-stacks}
  Let $X$ be an algebraic stack.
  \begin{enumerate}
  \item \label{I:fields-to-stacks:qaff} If $X$ has quasi-affine
    diagonal, then $X$ is affine-pointed.
  \item \label{I:fields-to-stacks:local} Let $g\colon V \to X$ be a
    quasi-finite and faithfully flat morphism of finite presentation
    (not necessarily representable). If $V$ is affine-pointed, then
    $X$ is affine-pointed.
  \end{enumerate}  
\end{lemma}
\begin{proof}
  Throughout, we fix a field $k$ and a morphism $x\colon \Spec k \to
  X$.

  For \ref{I:fields-to-stacks:qaff}, since $k$ is a field,
  every extension in $\QCoh(\Spec k)$ is split; thus
  $x_*$ is cohomologically affine
  \cite[Def.~3.1]{2008arXiv0804.2242A}. Since $X$ has quasi-affine
  diagonal, this property is preserved after pulling back $x$ along a
  smooth morphism $p\colon U \to X$, where $U$ is an affine scheme
  \cite[Prop.~3.10(vii)]{2008arXiv0804.2242A}. By Serre's Criterion
  \cite[II.5.2.2]{EGA}, the morphism $\Spec k \times_X
  U \to U$ is affine; and this case follows.

  For \ref{I:fields-to-stacks:local}, the pullback of $g$ along $x$
  gives a quasi-finite and faithfully flat morphism $g_0 \colon V_0
  \to \Spec k$. Since $V_0$ is discrete with finite stabilizers,
  there exists a finite surjective morphism $W_0\to V_0$ where $W_0$ is a
  finite disjoint union of spectra of fields. By assumption $W_0\to V_0\to V$ is
  affine; hence so is $V_0\to V$ (by Chevalley's Theorem
  \cite[Thm.~8.1]{rydh-2009} applied smooth-locally on $V$). By descent, $\Spec k \to X$ is affine and
  the result follows.
\end{proof}
The following lemma highlights the benefits of affine-pointed stacks.
\begin{lemma}\label{L:aff-pointed-use}
  Let $f_1$, $f_2 \colon T\to X$ be morphisms of algebraic stacks and
  let $\gamma\colon f_1^*\Rightarrow f_2^*$ be a natural
  transformation of cocontinuous tensor functors. If $X$ is
  affine-pointed, then the induced maps of topological spaces $|f_1|$, $|f_2|
  \colon |T| \to |X|$ coincide.
\end{lemma}
\begin{proof}
  It suffices to prove that if $T=\Spec k$, where $k$ is a field, then
  $\gamma$ is realizable. Since $X$ is affine-pointed, the morphisms $f_1$ and $f_2$ are
  affine. Also, the natural
  transformation $\gamma$ induces, by adjunction, a morphism of
  quasi-coherent $\sO_X$-algebras $\gamma^\vee(\sO_T)\colon (f_2)_*\sO_T \to
  (f_1)_*\sO_T$. In particular, $\gamma^\vee(\sO_T)$ induces a
  morphism $v\colon T \to T$ over $X$. We are now free to replace $X$
  by $T$, $f_2$ by $\id{T}$, and $f_1$ by $v$. Since $T$ is affine,
  the result now follows from Lemma \ref{L:qaff-tensorial:abs}.
\end{proof}
We can now prove the following proposition (generalizing Lurie's
corresponding result for an algebraic stack with affine diagonal).
\begin{proposition}\label{P:quaff_full_faithful}
  Let $X$ be an algebraic stack.
  \begin{enumerate}
      \item \label{I:quaff_ff:abs}If $T$ is an algebraic stack and $X$ has
    quasi-affine diagonal, then the functor $\TF{X}(T)$ is fully
    faithful.
  \item \label{I:quaff_ff:rel} Let $T$ be a quasi-affine scheme and
    let $f_1$, $f_2 \colon T \to X$ be quasi-affine morphisms.
    \begin{enumerate}
    \item\label{I:quaff_ff:rel_faith} If $\alpha$, $\beta \colon f_1
      \Rightarrow f_2$ are $2$-morphisms and $\alpha^* = \beta^*$ as
      natural transformations $f_1^* \Rightarrow f_2^*$, then $\alpha
      = \beta$.
    \item \label{I:quaff_ff:rel_full} Let $\gamma \colon f_1^*
      \Rightarrow f_2^*$ be a natural transformation of cocontinuous
      tensor functors.  If $\gamma$ is an isomorphism or $X$ is
      affine-pointed, then $\gamma$ is realizable.
    \end{enumerate}
  \end{enumerate}
\end{proposition}
\begin{proof}
  For \ref{I:quaff_ff:abs}, we may assume that $T$ is an affine scheme
  (Lemma \ref{L:fppf_local_tannakian}). Then
  every morphism $T \to X$ is quasi-affine and the result follows
  by~\ref{I:quaff_ff:rel} and
  Lemma~\ref{L:fields-to-stacks}\ref{I:fields-to-stacks:qaff}.

  For \ref{I:quaff_ff:rel},
  there are quasi-compact open immersions ${i_k \colon T
  \hookrightarrow V_k}$ over $X$, where $V_k :=
  \sSpec_{X}((f_k)_*\sO_T)$ and $k=1$, $2$.
  Let $v_k \colon V_k \to X$ be the induced
  $1$-morphism.

  We first treat \ref{I:quaff_ff:rel_faith}. The hypotheses imply that
  $\alpha_* = \beta_*$ as natural isomorphisms of functors from
  $(f_2)_*$ to $(f_1)_*$. In particular, $\alpha_*$ and $\beta_*$
  induce the same $1$-morphism from $V_1$ to $V_2$ over $X$. Since
  $i_1$ and $i_2$ are open immersions, they are monomorphisms; hence
  $\alpha=\beta$.

  We now treat \ref{I:quaff_ff:rel_full}. The natural transformation $\gamma\colon f_1^* \Rightarrow f_2^*$
  uniquely induces a natural transformation of lax symmetric monoidal functors
  $\gamma^\vee \colon (f_2)_* \Rightarrow (f_1)_*$. In particular, there is an induced morphism of quasi-coherent $\sO_X$-algebras
  $\gamma^\vee(\sO_T) \colon (f_2)_*\sO_T \to (f_1)_*\sO_T$; hence a
  morphism of algebraic stacks $g\colon V_1 \to V_2$ over $X$. Note
  that $\gamma^\vee$ uniquely induces a natural transformation of
  lax symmetric monoidal functors
  $(i_2)_* \Rightarrow g_*(i_1)_*$, and by adjunction we have a
  uniquely induced natural transformation of tensor functors
  $\gamma' \colon (g\circ i_1)^* \Rightarrow i_2^*$.

  Replacing $X$ by $V_2$, $f_1$ by $g\circ i_1$, $f_2$ by $i_2$, and $\gamma$ by $\gamma'$, 
  we may assume that $f_2$ is a quasi-compact open
  immersion such that $\sO_X \to (f_2)_*\sO_T$ is an isomorphism. 

  If $\gamma$ is an isomorphism, then $f_1$ is also a quasi-compact open immersion. Let $Z_1$ and $Z_2$ denote closed substacks of $X$ whose complements are $f_1(T)$ and $f_2(T)$, respectively. Then $f_1^*\sO_{Z_2} \cong f_2^*\sO_{Z_2} \cong 0$, so $f_1(T) \subseteq f_2(T)$. Arguing similarly, we obtain the reverse inclusion and we see that $f_1(T) = f_2(T)$. Since $f_1$ and $f_2$ are open immersions, we obtain the result when $\gamma$ is assumed to be an isomorphism.

  Otherwise, Lemma \ref{L:aff-pointed-use} implies that $f_1$
  factors through $f_2(T) \subseteq X$. We may now replace $X$ by
  $T$ and $\gamma$ with $(f_2)_*(\gamma)$~\cite[Prop.~2.3.6]{MR3144607}.
  Then $X$ is quasi-affine and the result follows from Lemma
  \ref{L:qaff-tensorial:abs}.
\end{proof}
From Proposition \ref{P:quaff_full_faithful}\ref{I:quaff_ff:rel_full}, we obtain an analogue of Lemma \ref{L:aff-pointed-use} for natural isomorphisms of functors when $X$ has affine stabilizers (as opposed to affine-pointed). 
\begin{corollary}\label{C:points-affine-stabs}
  Let $f_1$, $f_2 \colon T\to X$ be morphisms of algebraic stacks and
  let $\gamma\colon f_1^*\simeq f_2^*$ be a natural isomorphism of
  cocontinuous tensor functors. If $X$ has affine stabilizers and
  quasi-compact diagonal, then the induced maps of
  topological spaces $|f_1|$, $|f_2| \colon |T| \to |X|$ coincide.
\end{corollary}
\begin{proof}
  It suffices to prove the result when $T=\Spec k$, where $k$ is a
  field. Since $X$ has affine stabilizers and quasi-compact diagonal
  the morphisms $f_1$ and $f_2$ are quasi-affine
  \cite[Thm.~B.2]{MR2774654}. The result now follows from Proposition
  \ref{P:quaff_full_faithful}\ref{I:quaff_ff:rel_full}.
\end{proof}
  The following result, in a slightly different context, was proved by Sch\"appi \cite[Thm.~1.3.2]{2012arXiv1206.2764S}. Using the Totaro--Gross theorem, we can simplify Sch\"appi's arguments in the algebro-geometric setting.
  \begin{theorem}\label{T:resprop-tensorial}
    Let $X$ be a quasi-compact and quasi-separated algebraic stack
    with affine stabilizers. If $X$ has the resolution property, then it is
    tensorial.
  \end{theorem}
  \begin{proof}
    Let $T$ be an algebraic stack. By
    Totaro--Gross~\cite{2013arXiv1306.5418G}, there is a quasi-affine
    morphism $g\colon X\to \BGL_{N,\Z}$. By Lemma \ref{L:qaff-tensorial:rel},
    it is enough to consider $X=\BGL_{N,\Z}$. Note that $X$ is
    quasi-compact with affine diagonal, so the functor
    $\TF{X}(T)$ is fully faithful (Proposition
    \ref{P:quaff_full_faithful}). It remains to prove that every
    cocontinuous tensor functor $f^*\colon \QCoh(X) \to \QCoh(T)$ is
    algebraic.

    Let $p\colon \Spec \Z\to\BGL_{N,\Z}$ be the universal
    $\GL_N$-bundle and let $\sA=p_*\Z$ be the regular
    representation. There is an exact sequence
    \[
    0\to \sO_{\BGL_{N,\Z}}\to \sA \to \sQ\to 0
    \]
    of flat quasi-coherent sheaves. Write $\sA$ as the directed colimit of its
    subsheaves $\sA_\lambda$ of finite type containing the unit and let
    $\sQ_\lambda = \sA_\lambda/\sO_{\BGL_{N,\Z}}\subseteq \sQ$. Then
    $\sA_\lambda$ and $\sQ_\lambda$ are vector bundles.

    It is well-known that (1) any tensor
    functor $f^*\colon \QCoh(X)\to \QCoh(T)$ preserves dualizable objects and
    exact sequences of dualizable
    objects (for example, see \cite[Def.~4.7.1 \& Lem.~4.7.10]{brandenburg_thesis}) and
    (2) the dualizable
    objects in $\QCoh(Y)$ are the vector bundles for any algebraic stack
    $Y$~\cite[Prop.~4.7.5]{brandenburg_thesis}. We thus have exact sequences
    \[
    0\to \sO_T\to f^*\sA_\lambda \to f^*\sQ_\lambda\to 0
    \]
    of vector bundles. Since $f^*$ is cocontinuous, we also obtain an
    exact sequence
    \[
    0\to \sO_T\to f^*\sA \to f^*\sQ\to 0
    \]
    of flat quasi-coherent sheaves. In particular,
    $f^*\sA$ is a faithfully flat algebra.

    Let
    $V=\sSpec_T(f^*\sA)$; then $r\colon V\to T$ is faithfully flat. By
    Corollary~\ref{C:pb_affine_stack}, we have a
    cocartesian diagram
    \[
    \xymatrix{\QCoh(V) & \QCoh(\Spec \Z) \ar[l]_-{f'\mathrlap{^*}} \\
      \QCoh(T) \ar[u]^{r^*} & \QCoh(X)\ar[u]_{p^*} \ar[l]_{f\mathrlap{^*}}.}
    \]
    Since $\Spec \Z$ is tensorial (Lemma~\ref{L:qaff-tensorial:abs}), the
    functor $f'^*$ is algebraic. Thus, $f'^*p^*\simeq r^*f^*$ is algebraic. Descent 
    along $r\colon V \to T$ (Lemma 
\ref{L:fppf_local_tannakian}\ref{LI:fppf_local_tannakian:alg}) implies that $f^*$ is algebraic.
  \end{proof}


\section{Tensor localizations}\label{S:tensor-localization}
The goal of this section is to prove the following theorem. 
\begin{theorem}\label{T:tensor-loc-for-algebraic-stacks}
Let $X$ be a quasi-compact and quasi-separated algebraic stack.  Let $i\colon Z \to X$ be a finitely presented closed immersion defined by an ideal sheaf $I$. Let $j\colon U \to X$ be the open complement of $Z$. Let $T$ be
an algebraic stack and let $f^*\colon \QCoh(X)\to \QCoh(T)$ be a cocontinuous
tensor functor. Let $i_T \colon Z_T \to T$ be the closed immersion defined by the ideal $I_T:= \Im(f^*I \to \sO_T)$. Let $j_T \colon U_T \to T$ denote the complement of $Z_T$.
\begin{enumerate}
\item \label{TI:tensor-loc-for-algebraic-stacks:open} There exists a cocontinuous tensor functor
\[
f_U^*\colon \QCoh(U)\to \QCoh(U_T),
\]
which is essentially unique, and a canonical 
  isomorphism of tensor functors $j_T^*f^*\simeq f_U^*j^*$.
\item \label{TI:tensor-loc-for-algebraic-stacks:closed} For each integer $n\geq 0$,
\[
f_{Z^{[n]}}^*:=(i_{T}^{[n]})^*f^*(i^{[n]})_*\colon \QCoh(Z^{[n]})\to \QCoh(Z_T^{[n]})
\]
 is a cocontinuous
  tensor functor and there is a canonical isomorphism of tensor functors
  $(i_T^{[n]})^*f^*\simeq f_{Z^{[n]}}^*(i^{[n]})^*$. Moreover, $f^*(i^{[n]})_*\simeq (i_T^{[n]})_*f_{Z^{[n]}}^*$.
\end{enumerate}
In addition, if $f^*$ preserves sheaves of finite type, then the same is true of $f^*_U$ and $f_{Z^{[n]}}^*$ for all $n\geq 0$. 
\end{theorem}
Theorem \ref{T:tensor-loc-for-algebraic-stacks} features in a key way in the proof of our 
main theorem (Theorem \ref{T:main_tannaka_non-noeth}), which we prove via 
stratifications and formal gluings. From this context, we hope that the long and technical statement of Theorem \ref{T:tensor-loc-for-algebraic-stacks} should appear to be quite natural. While Theorem 
\ref{T:tensor-loc-for-algebraic-stacks}\ref{TI:tensor-loc-for-algebraic-stacks:closed} 
follows easily from the results of \S\ref{SS:modules}, Theorem 
\ref{T:tensor-loc-for-algebraic-stacks}\ref{TI:tensor-loc-for-algebraic-stacks:open} is 
more subtle. It turns out, however, that it is a consequence of a more general result about 
Grothendieck abelian tensor categories (Theorem \ref{T:localization-factorization}), which is what we will spend most of this section proving.

Let $\cC$ be a Grothendieck abelian category. A \emph{Serre subcategory} is a
full non-empty subcategory $\cK\subseteq \cC$ closed under taking
subquotients and extensions. Serre
subcategories are abelian and the inclusion functor is exact. A Serre subcategory is \emph{localizing} if it is also closed under small direct sums in $\cC$, equivalently, it is closed under small colimits in $\cC$.

If $\cK \subseteq \cC$ is a Serre subcategory, then there is a quotient $\cQ$ of $\cC$ by
$\cK$ and an exact functor $q^*\colon \cC \to \cQ$, which is universal
for exact functors out of $\cC$ that vanish on $\cK$ \cite[Ch.~III]{MR0232821}. Note that $\cK$ is localizing if and only if the quotient
$q^*\colon \cC \to \cQ$ is a \emph{localization}, that is, $q^*$ admits a right
adjoint $q_* \colon \cQ \to \cC$; it follows that $\cQ$ is Grothendieck abelian, $q^*$ is cocontinuous, $q_*$ is fully
faithful and $q^*q_*\simeq \id{\cQ}$. This statement follows by combining the Gabriel--Popescu Theorem (e.g., \cite[Thm.~6.25]{MR0236236}) with \cite[Prop.~6.21]{MR0236236}.

Let $\cC$ be a Grothendieck abelian tensor category and let $\cK \subseteq \cC$ be a Serre subcategory. We say that $\cK$ is a \emph{tensor ideal} if $\cK$ is closed under
tensor products with objects in $\cC$. If $\cK$ is also localizing, then we say that $\cK$ is a \emph{localizing tensor ideal}.

If $f^*\colon \cC \to \cD$ is an exact cocontinuous tensor functor between Grothendieck abelian tensor categories, then
$\ker(f^*)$ is a localizing tensor ideal. Conversely, if $\cK\subseteq \cC$ is a localizing
tensor ideal, then the quotient $\cQ=\cC/\cK$ is a
Grothendieck abelian tensor
category, the localization $q^*\colon \cC\to \cQ$ is an exact cocontinuous
tensor functor and $\ker (q^*) = \cK$; in this situation, we will refer to $q^*$ as a \emph{tensor
  localization}. 

\begin{example}\label{E:quotab_alg_stack}
  Let $\map{f}{X}{Y}$ be a morphism of algebraic stacks. If $f$ is
  flat, then $f^*$ is exact. If $f$ is a quasi-compact flat
  monomorphism (e.g., a quasi-compact open immersion), then $\QCoh(X)$
  is the quotient of $\QCoh(Y)$ by $\ker(f^*)$. This follows from the
  fact that the counit $f^*f_*\to \id{}$ is an isomorphism so that
  $f_*$ is a section of $f^*$ \cite[Prop.~III.2.5]{MR0232821}.
\end{example}

\begin{definition}
  Let $\cC$ be a Grothendieck abelian tensor category. For $M\in \cC$ let
  $\map{\varphi_M}{\sO_{\cC}}{\sHom_{\cC}(M,M)}$ denote the adjoint to
  the canonical isomorphism $\sO_{\cC}\otimes_{\cC} M\to M$. Let the
  annihilator $\Ann_{\cC}(M)$ of $M$ be the kernel of $\varphi_M$,
  which we consider as an ideal of $\sO_{\cC}$.
\end{definition}

\begin{example}\label{E:qcohstk-localization-sequence} 
  Let $X$ be an algebraic stack and let $\sF\in \QCoh(X)$. Then
  $\Ann_{\QCoh(X)}(\sF)=\QC\bigl(\Ann_{\Mod(X)}(\sF)\bigr)$. In
  particular, if $\sF$
  is of finite type, then $\Ann_{\QCoh(X)}(\sF) = \Ann_{\Mod(X)}(\sF)$. 
\end{example}

Recall that an object $c\in \cC$ is \emph{finitely generated} if the
natural map:
\[
\varinjlim_\lambda \Hom_{\cC}(c,d_\lambda) \to
\Hom_{\cC}(c,\varinjlim_\lambda d_\lambda)
\]
is bijective for every direct system $\{d_\lambda\}_{\lambda}$ in $\cC$ with
monomorphic bonding maps. A category $\cC$ is \emph{locally finitely generated}
if it is cocomplete (all small colimits exist) and has a set $\mathcal{A}$ of finitely generated objects
such that every object $c$ of $\cC$ is a directed colimit of objects from
$\mathcal{A}$.
\begin{example}\label{E:stk_fg_description}
  Let $X$ be a quasi-compact and quasi-separated algebraic stack. The
  finitely generated objects in $\QCoh(X)$ are the quasi-coherent
  sheaves of finite type. Thus $\QCoh(X)$ is locally finitely
  generated~\cite{rydh-2014}.
\end{example}
We also require the following definition.
\begin{definition}
  Let $q^*\colon \cC \to \cQ$ be a tensor localization. Then it is \emph{supported} if $q^*(\sO_{\cC}/\Ann(K)) \cong 0$ 
  for every finitely generated object $K$ of $\cC$ such that $q^*(K) \cong 0$. 
\end{definition}
The notion of a supported tensor localization is very natural. 
\begin{example}\label{E:qcohstk-localization-factorization}
  If $f \colon X \to Y$ is a flat monomorphism of
  quasi-compact and quasi-separated algebraic stacks, then the tensor localization $f^*\colon \QCoh(Y) \to
  \QCoh(X)$ of Example \ref{E:quotab_alg_stack} is supported. Indeed,
  if $M$ is a quasi-coherent $\sO_Y$-module of finite type in the kernel
  of $f^*$, then
  $f^*\Ann_{\QCoh(Y)}(M) = \Ann_{\QCoh(X)}(f^*M) = \sO_X$.
\end{example}
We now have our key result, which also
generalizes~\cite[Lem.~3.3.6]{MR3144607}.
\begin{theorem}\label{T:localization-factorization}
  Let $\cC$ be a locally finitely generated Grothendieck abelian tensor category
  and let $\map{q^*}{\cC}{\cQ}$ be a supported tensor localization. Let
  $\cD$ be a Grothendieck abelian tensor category. If
  $\map{f^*}{\cC}{\cD}$ is a cocontinuous tensor functor such that $f^*(K)
  \cong 0$ for every finitely generated object $K$ of $\cC$ such that
  $q^*(K)\cong 0$, then $f^*$ factors essentially uniquely through
  a cocontinuous tensor functor $\map{g^*}{\cQ}{\cD}$. If $f^*$ preserves
  finitely generated objects, then so does $g^*$.
\end{theorem}
Note that Theorem \ref{T:localization-factorization} is trivial if $f^*$ is exact. The challenge is to
use the symmetric monoidal structure to deduce this also when $f^*$ is merely
right-exact. The proof we give is a straightforward generalization of
\cite[Lem.~3.3.6]{MR3144607}. First, we will see how Theorem \ref{T:localization-factorization} implies Theorem \ref{T:tensor-loc-for-algebraic-stacks}.
\begin{proof}[Proof of Theorem \ref{T:tensor-loc-for-algebraic-stacks}]
  For \ref{TI:tensor-loc-for-algebraic-stacks:closed}, note that $(i^{[n]})_*$ identifies $\QCoh(Z^{[n]})$ with the category of
  modules over the algebra $A_n=\sO_X/I^{n+1}$. The algebra $f^*A_n$ is
  $\sO_T/I_T^{n+1}$ and $(f_{Z^{[n]}})^*=(f_{A_n})^*$ in the terminology of
  \S\ref{SS:modules}.

  For \ref{TI:tensor-loc-for-algebraic-stacks:open}, recall that $\QCoh(X)$ is locally finitely generated
  (Example~\ref{E:stk_fg_description}) and that $j^*\colon \QCoh(X)\to
  \QCoh(U)$ is a supported localization
  (Example~\ref{E:qcohstk-localization-factorization}). If $K\in \QCoh(X)$ is
  finitely generated and $j^*K=0$, then $I^m K=0$ for sufficiently large $m$.
  Thus, the natural map $I^m \otimes_{\sO_X} K\to \sO_X\otimes_{\sO_X} K \cong
  K$ is zero. Applying $j_T^*f^*$, the map becomes the identity since
  $j_T^*f^*(I^m)\to j_T^*f^*(\sO_X)=\sO_{U_T}$ is an isomorphism. It follows that
  $j_T^*f^*K=0$. We may thus apply
  Theorem~\ref{T:localization-factorization} and deduce that $j_T^*f^*$ factors
  via $j^*$ and a tensor functor $f_U^* \colon \QCoh(U)\to \QCoh(U_T)$.
 \end{proof}
To prove
Theorem~\ref{T:localization-factorization} we require the following
lemma.
\begin{lemma}[{\cite[Lem.~3.3.2]{MR3144607}}]\label{L:ideal-iso}
  Let $\map{f^*}{\cC}{\cD}$ be a cocontinuous tensor functor.  If
  $I\subseteq \sO_{\cC}$ is an $\sO_{\cC}$-ideal such that
  $f^*(\sO_{\cC}/I)\cong 0$, then $f^*(I)\to
  f^*(\sO_{\cC})$ is an isomorphism.
\end{lemma}
\begin{proof}
Since $f^*$ is right-exact and $f^*(\sO_{\cC}/I)=0$, it follows that $f^*(I)\to
f^*(\sO_{\cC})=\sO_{\cD}$ is surjective. Let $J=f^*(I)$ and let
$\varphi\colon J\to \sO_{\cD}$ denote the surjection. The multiplication
$I\otimes_{\cC} I\to I$ factors through $I\otimes_{\cC} \sO_{\cC}$ and
$\sO_{\cC}\otimes_{\cC} I$ and gives rise to the commutative diagram
\[
\xymatrix@C+5mm{%
J\otimes_\cD J\ar@{->>}[r]^-{\id{J}\otimes \varphi}\ar@{->>}[d]_{\varphi\otimes \id{J}} & J\otimes_\cD \sO_{\cD}\ar[d]^{\iso}\\
\sO_{\cD}\otimes_\cD J\ar[r]^{\iso} & J
}
\]
Let $\eta_F$ denote
the unit of the adjunction between $-\otimes_{\cD} F$ and $\sHom_{\cD}(F,-)$.
Then we obtain the commutative diagram
\[
\xymatrix@C+10mm{%
J\ar[r]^-{\eta_J(J)}\ar@{->>}[d]_{\varphi} & \sHom_\cD(J,J\otimes J)\ar[r]^-{\sHom(-,\id{J}\otimes \varphi)}\ar[d]^{\sHom(-,\varphi\otimes\id{J})} & \sHom_\cD(J,J\otimes \sO_{\cD})\ar[d]^{\iso}\\
\sO_\cD\ar[r]^-{\eta_J(\sO_{\cD})} & \sHom_\cD(J,\sO_{\cD}\otimes J)\ar[r]^{\iso} & \sHom_\cD(J,J).
}
\]
But the top row also factors as
\[
\xymatrix@C+5mm{%
J\ar[r]^-{\eta_{\sO_\cD}(J)}
 & \sHom_\cD(\sO_{\cD},J\otimes \sO_{\cD})\ar[r]^{\sHom(\varphi,-)}
 & \sHom_\cD(J,J\otimes \sO_{\cD})
}
\]
which is injective since $\eta_{\sO_\cD}$ is an isomorphism and
$\varphi$ is surjective. It follows
that $J\to \sHom_\cD(J,J)$ is injective, hence so is $\varphi\colon J\to \sO_{\cD}$.
\end{proof}

\begin{proof}[Proof of Theorem~\ref{T:localization-factorization}]
  If $K \in \cC$, since $\cC$ is locally finitely generated, it may be written
  as a directed colimit $K = \varinjlim_\lambda K_\lambda$, where $K_\lambda
  \subseteq K$ and $K_\lambda$ is finitely generated. If $K \in \ker (q^*)$,
  then $q^*K_\lambda \subseteq q^*K \cong 0$. In particular, $\cK:=\ker(q^*)
  \subseteq
  \ker(f^*)$.

  Let $0\to K\to M\to N\to Q\to 0$ be an exact sequence in
  $\cC$ with $K,Q\in \cK$. We have to prove that $f^*(M\to N)$ is an
  isomorphism in $\cD$. Let $N_0$ be the image of $M$ in $N$.  By
  right-exactness, we have an exact sequence $f^*(K)\to f^*(M)\to
  f^*(N_0)\to 0$.  Since $f^*(K)=0$, we have that
  $f^*(M)=f^*(N_0)$. We may thus replace $M$ with $N_0$ and assume
  that $K=0$ and $M\to N$ is injective.

  Write $N$ as the directed colimit of finitely generated subobjects
  $N^\circ_\lambda\subseteq N$. Let
  $N_\lambda=M+N^\circ_\lambda\subseteq N$ and
  $I_\lambda=\Ann(N_\lambda/M)$.  By definition,
  we have that $I_\lambda\otimes N_\lambda/M\to N_\lambda/M$ is zero;
  hence $I_\lambda\otimes N_\lambda\to N_\lambda$ factors through
  $M$.

  Note that $N_\lambda/M=N^\circ_\lambda/(N^\circ_\lambda\cap M)$
  is a quotient of a finitely generated object and a subobject of
  $Q$, so $\sO_{\cC}/I_\lambda \in \cK$ since $q^*$ is supported.
  We conclude that
  $f^*(I_\lambda)\to f^*(\sO_{\cC})$ is an isomorphism using
  Lemma~\ref{L:ideal-iso}.
  Now consider the commutative diagrams:
  \[
  \vcenter{\xymatrix{
      I_\lambda\otimes M\ar[r]\ar[d] & M\ar[d] \\
      I_\lambda\otimes N_\lambda\ar[r]\ar[ur] & N_\lambda }}\text{ and
  }\vcenter{\xymatrix{
      f^*(M)\ar[r]^{\iso}\ar[d] & f^*(M)\ar[d] \\
      f^*(N_\lambda)\ar[r]^{\iso}\ar[ur] & f^*(N_\lambda), }}
  \]
  where the right diagram is obtained by applying $f^*$ to the left
  diagram.  It follows that $f^*(M)\to f^*(N_\lambda)$ is an
  isomorphism. Since $f^*$ is cocontinuous, it follows that $f^*(M)\to
  f^*(N)=\varinjlim f^*(N_\lambda)$ is an isomorphism.

  This proves that $f^*=g^*q^*$ where $g^*=f^*q_*$. It is readily verified that
  $g^*$ is cocontinuous (it preserves small direct sums and is right-exact).
  If $M\in \cQ$ is a finitely generated object, then we may find a
  finitely generated object $N\in \cC$ such that $M=q^*N$. Indeed,
  by assumption $q_*M$ is a filtered colimit of finitely generated objects.
  It follows that there is a finitely generated subobject $N\subseteq q_*M$
  such that $q^*N\to M$ is an isomorphism. If $f^*$ preserves finitely
  generated objects, then $g^*M=f^*N$ is finitely generated.
\end{proof}

To show how powerful tensor localization is, we can quickly prove that tensoriality is local for the Zariski topology---even for stacks. 
\begin{theorem}\label{T:scheme-tensorial}
Let $X$ be a quasi-compact and quasi-separated algebraic stack. Let
$X=\bigcup_{k=1}^n X_k$ be an open covering by quasi-compact open substacks. If
every $X_k$ is tensorial, then so is $X$.
\end{theorem}
\begin{proof}
Let $j_k\colon X_k\to X$ denote the open immersion and let $I_k$ be an ideal
of finite type defining a closed substack complementary to
$X_k$~\cite[Prop.~8.2]{rydh-2014}.

Let $T$ be an algebraic stack. First we will show that $\TF{X}(T)$ is fully faithful.
Thus, let $f,g\colon T\to X$ be two morphisms and suppose that
we are given a natural transformation of cocontinuous tensor functors $\gamma\colon f^* \Rightarrow g^*$. Then
$f^*(\sO_X/I_k) \surj g^*(\sO_X/I_k)$ so there is an inclusion $f^{-1}(X_k)\subseteq g^{-1}(X_k)$
for every $k$. Let $T_k=f^{-1}(X_k)$, let $j_{k,T}\colon T_k\to T$ denote the
corresponding open immersion and let $f_k,g_k\colon T_k\to X_k$ denote the
restrictions of $f$ and $g$. Since
$(f_k)^*=j_{k,T}^*f^*(j_k)_*$ and $(g_k)^*=j_{k,T}^*g^*(j_k)_*$, we obtain
a natural transformation $\gamma_k\colon f_k^* \Rightarrow g_k^*$, hence a
unique $2$-isomorphism $f_k\Rightarrow g_k$. Since $T=\bigcup_{k=1}^N T_k$, it
follows by fppf-descent, that $\TF{X}(T)$ is faithful (Lemma~\ref{L:fppf_local_tannakian}\ref{LI:fppf_local_tannakian:faith}). As this holds for all $T$,
we also have that $\TF{X}(T_k\cap T_{k'})$ is faithful and it follows by
fppf-descent that $\TF{X}(T)$ is full
(Lemma~\ref{L:fppf_local_tannakian}\ref{LI:fppf_local_tannakian:ff}).

For essential surjectivity, let $f^*\colon \QCoh(X)\to \QCoh(T)$ be a
cocontinuous
tensor functor. The surjection $\sO_T\surj f^*(\sO_X/I_k)$ defines a closed
subscheme and we let $j_{k,T}\colon T_k\to T$ denote its open complement. By Theorem
\ref{T:tensor-loc-for-algebraic-stacks}\ref{TI:tensor-loc-for-algebraic-stacks:open}, 
$j_{k,T}^*f^*$ factors via $j_k^*$ and a tensor functor $f_{k}^* \colon \QCoh(X_k)\to 
\QCoh(T_k)$. The latter is
algebraic by assumption; hence, so is $j_{k,T}^*f^*=f_k^*j_k^*$.

Finally, since $\sO_X/I_1\otimes \dots \otimes \sO_X/I_n=0$, it follows
that $f^*(\sO_X/I_1)\otimes \dots \otimes f^*(\sO_X/I_n)=0$ so
$T=\bigcup_{k=1}^n T_k$ is an open covering. We conclude that $f^*$ is
algebraic by fppf descent
(Lemma~\ref{L:fppf_local_tannakian}\ref{LI:fppf_local_tannakian:alg}).
\end{proof}
Combining Theorem \ref{T:scheme-tensorial} with Lemma \ref{L:qaff-tensorial:abs}
we obtain a short proof of the main result of \cite{MR3144607}.
\begin{corollary}[Brandenburg--Chirvasitu]
Every quasi-compact and quasi-separated scheme is tensorial.
\end{corollary}


\section{The Main Lemma}\label{S:main-lemma}
The main result of this section is the following technical lemma, which proves that the tensorial property extends over nilpotent thickenings of quasi-compact algebraic stacks with affine stabilizers having the resolution property.
\begin{lemma}[Main Lemma]\label{L:funny_nil_res}
  Let $i\colon X_0 \to X$ be a closed immersion of algebraic stacks
  defined by a quasi-coherent ideal $I$ such that $I^n = 0$ for some
  integer $n>0$. Suppose that $X_0$ is quasi-compact and quasi-separated with
  affine stabilizers. If $X_0$ has the resolution property, then $X$
  is tensorial.
\end{lemma}
We have another lemma that will be crucial for proving Lemma
\ref{L:funny_nil_res}.
\begin{lemma}\label{L:res_strange}
  Consider a $2$-cocartesian diagram of algebraic stacks:
  \[
  \xymatrix{U_0 \ar[d]_{p_0} \ar@{^(->}[r]^i &  U \ar[d]^p \\ X_0 \ar@{^(->}[r]^j & X,}
  \]
  such that the following conditions are satisfied.
  \begin{enumerate}
  \item $i$ is a nilpotent closed immersion;
  \item $U_0$ is an affine scheme; and
  \item $X_0$ is quasi-compact and quasi-separated with affine stabilizers. 
  \end{enumerate}
  If $X_0$ has the resolution property, then so has $X$. 
\end{lemma}
\begin{proof}
  Note that $X_0$ has affine diagonal by the Totaro--Gross theorem; hence
  $p_0$ is affine.
  By \cite[Prop.~A.2]{hallj_openness_coh}, the square is a geometric
  pushout. In particular, $j$ is a nilpotent closed immersion, $p$ is
  affine, and the natural map $\sO_X \to p_*\sO_{U}
  \times_{p_*i_*\sO_{U_0}} j_*\sO_{X_0}$ is an isomorphism. By the
  Totaro--Gross Theorem \cite[Cor.~5.9]{2013arXiv1306.5418G}, there
  exists a vector bundle $V_0$ on $X_0$ such that the total space of
  the frame bundle of $V_0$ is quasi-affine. Let $E_0 = p_0^*V_0$;
  then, since $U_0$ is affine, there exists a vector bundle $E$ on $U$
  equipped with an isomorphism $\alpha\colon i^*E \to E_0$. Let $V$ be
  the quasi-coherent $\sO_X$-module $p_*E \times_\alpha j_*V_0$. By
  \cite[Thm.~2.2(iv)]{MR2044495}, $V$ is a vector bundle on $X$ and
  there is an isomorphism $j^*V \cong V_0$. By
  \cite[Prop.~5.7]{2013arXiv1306.5418G}, it follows that $X$ has the
  resolution property.  
\end{proof}
\begin{proof}[Proof of Lemma \ref{L:funny_nil_res}]
  We prove the result by induction on $n> 0$. The case $n=1$ is Theorem 
  \ref{T:resprop-tensorial}. So we let $n>1$ be an integer and we will assume that if 
  $W_0 \hookrightarrow W$ is any closed immersion of algebraic stacks defined by an ideal 
  $J$ such that $J^{n-1}=0$ and $W_0$ has the resolution property, then $W$ is 
  tensorial. We now fix a closed immersion of algebraic stacks $i\colon X_0 \to X$ defined 
  by an ideal $I$ such that $I^n = 0$ and $X_0$ has the resolution property. It remains to 
  prove that $X$ is tensorial. 
  
  We observe that
  the Totaro--Gross Theorem \cite[Cor.~5.9]{2013arXiv1306.5418G}
  implies that $X_0$ has affine diagonal; thus, $X$ has affine diagonal. We
  have seen that $\TF{X}(T)$ is fully faithful
  (Proposition \ref{P:quaff_full_faithful}) so it remains to prove that
  $\TF{X}(T)$ is essentially surjective. By descent,
  it suffices to prove that if $T$ is an affine scheme and $f^* \colon
  \QCoh(X) \to \QCoh(T)$ is a cocontinuous tensor functor, then there
  exists an \'etale and surjective morphism $c \colon T' \to T$
  such that $c^*f^*$ is algebraic
  (Lemma~\ref{L:fppf_local_tannakian}\ref{LI:fppf_local_tannakian:alg}).
  
  By Corollary~\ref{C:pb_affine_stack}, there is a $2$-cocartesian
  diagram in $\GATC$
  \[
  \xymatrix{\QCoh(T_0) & \QCoh(X_0) \ar[l]_{f_0^*} \\
    \QCoh(T) \ar[u]^{k^*} & \QCoh(X)\ar[u]_{i^*} \ar[l]_{f^*},}
  \]
  where $k\colon T_0\to T$ is the closed immersion defined by the image
  $K$ of $f^*I$
  in $\sO_T$. In particular, $K^n = 0$. Since $X_0$ has the resolution property, $f_0^*$ is
  given by a morphism of algebraic stacks $f_0 \colon T_0 \to
  X_0$ (Theorem~\ref{T:resprop-tensorial}).
  
  Let $p\colon U \to X$ be a smooth and surjective morphism, where $U$
  is an affine scheme; then, $p$ is affine. The pullback of $p$ along
  the morphism $i\circ f_0 \colon T_0 \to X$ results in a smooth and
  affine surjective morphism of schemes $q_0 \colon V_0 \to T_0$. By
  \cite[IV.17.16.3(ii)]{EGA}, there exists an affine \'etale and
  surjective morphism $c_0 \colon T'_0 \to T_0$ such that the pullback
  $q_0' \colon V_0' \to T_0'$ of $q_0$ to $T_0'$ admits a section. By
  \cite[IV.18.1.2]{EGA}, there exists a unique affine \'etale morphism
  $c\colon T' \to T$ lifting $c_0\colon T'_0 \to T_0$. After replacing
  $T$ with $T'$ and $f^*$ with $c^*f^*$, we may thus assume that
  $q_0$ admits a section (Lemma 
\ref{L:fppf_local_tannakian}\ref{LI:fppf_local_tannakian:alg}).

  Let $X'=\sSpec_{X}(f_*\sO_{T})$. Let $I'=I(f_*\sO_{T})$
  be the $\sO_{X'}$-ideal generated by $I$ and let $X_0'=V(I')$. Then
  $X'$ is a quasi-compact stack with affine diagonal,
  $X'_0 \to X'$ is a closed immersion defined by an ideal whose $n$th
  power vanishes and $X_0'$ has the resolution property.
  Let $f'^* =
  \bar{f}^* \colon \QCoh(X') \to \QCoh(T)$ be the resulting
  tensor functor.

  Since $f'^*$ is right-exact, it follows that $K=\mathrm{im}(f'^*I' \to
  \sO_{T})$. Also, $I' \subseteq f'_*K \subseteq \sO_{X'}$.
  Thus $V(f'_*K) \subseteq X_0'$, so has the resolution property. Note that
  $f'^*I'\to f'^*f'_*K\to K$ is surjective.
  Since $f'_*$ is lax symmetric monoidal, for each integer $l\geq 1$ the morphism $(f'_*K)^{\otimes l} \to \sO_{X'}$ factors through $f'_*(K^{\otimes l}) \to \sO_{X'}$. In particular, $(f'_*(K^{l}))^2 \subseteq f'_*(K^{l+1})$ and $(f'_*K)^n =
  0$. We may thus replace $X$ by $X'$, $X_0$ by $V(f'_*K)$,
  $f^*$ by $f'^*$, $I$ by $f_*K$ and assume henceforth that
  \begin{enumerate}
    \item $\sO_X \to f_*\sO_T$ is an isomorphism,
    \item $I=f_*K$ for some $\sO_T$-ideal $K$ with $K^n = 0$,
    \item $f_*(K^l)^2 \subseteq f_*(K^{l+1})$ for each integer $l\geq 1$, 
    \item $(f_*K)^{l} \subseteq f_*(K^l)$ for $l\geq 1$, and
    \item $q_0 \colon V_0 \to T_0$ admits a section.
  \end{enumerate}
  
  For each integer $l\geq 0$ let $I_l = f_*(K^{l+1})$, which is a quasi-coherent sheaf of 
  ideals on $X$. Let $i_l \colon X_l \to X$ be the closed
  immersion defined by $I_l$ and let $k_{l} \colon T_l \to T$ be
  the closed immersion defined by $K^{l+1}$. Since
  $f^*f_*(K^{l+1}) \to f^*\sO_X=\sO_T$ factors through $K^{l+1}$,
  it follows that $k_{l}^*f^*(i_l)_*(\sO_{X_l})=\sO_{T_l}$. Hence,
  $f_l^*=k_{l}^*f^*(i_l)_*\colon \QCoh(X_l) \to \QCoh(T_l)$
  is a tensor functor and $k_l^*f^* \simeq f_l^*(i_l)^*$ (Theorem \ref{T:tensor-loc-for-algebraic-stacks}\ref{TI:tensor-loc-for-algebraic-stacks:closed}).

  By condition (iv), we see that $i_l \colon X_0 \to X_l$ is a closed immersion of algebraic 
  stacks defined by an ideal whose $(l+1)$th power is zero. In particular, if $l<n-1$, then 
  $X_l$ is tensorial by the inductive hypothesis. Thus, the tensor functor $f_l^*$ is 
  given by an affine morphism $f_l \colon T_l \to X_l$. 

  We will now
  prove by induction on $l\geq 0$ that $X_l$ has the resolution
  property. Since $X_{n-1} = X$, the result will then follow from Theorem 
  \ref{T:resprop-tensorial}. Note that (iii) implies that the closed immersion $X_l \to 
  X_{l+1}$ is a square zero extension of $X_l$ by $I_l/I_{l+1}$. Let $m=n-2$.
  \\*[2mm]
  \emph{Claim 1.} If $M \in \QCoh(T_m)$, the natural map $f_*(k_m)_*M \to p_*p^*f_*(k_m)_*M$ is split injective.\\*[1mm]
  \emph{Proof of Claim 1.} Form the 
  cartesian diagram of algebraic stacks:
  \[
  \xymatrix{V_0 \ar[d]_{q_0} \ar[r] & V_m \ar[d]_{q_m} \ar[r]^{g_m}
  & U_m \ar[d]_{p_m} \ar[r]^{u_m} & U \ar[d]^p \\ T_0 \ar[r] & T_m
  \ar[r]^{f_m} & X_m \ar[r]^{i_m} & X.}
  \]
  Now observe that $f_*(k_m)_*M \cong (i_m)_*(f_m)_*M$.  Since $f_m^*$ is given by a 
  morphism $f_m\colon T_m \to X_m$, there are natural isomorphisms:
  \begin{align*}
    p_*p^*f_*(k_m)_*M &\cong p_*p^*(i_m)_*(f_m)_*M \cong p_*(u_m)_*p_m^*(f_m)_*M\\
                      &\cong p_*(u_m)_*(g_m)_*q_m^*M \cong (i_m)_*(f_m)_*(q_m)_*q_m^*M.
  \end{align*}
  Hence, it remains to prove that the natural map $M \to (q_m)_*q_m^*M$ is split injective. 
  But $q_m$ is affine, so $(q_m)_*q_m^*M \cong 
  (q_m)_*\sO_{V_m}\otimes_{\sO_{T_m}}M$. Thus, we are reduced to proving that $\sO_{T_m} \to 
  (q_m)_*\sO_{V_m}$ is split injective. By (v), $q_0$ admits a section. Since
  $q_m$ is smooth and $T_m$ is affine, the section that $q_0$ admits
  lifts to a section of $q_m$. This implies that the morphism $\sO_{T_m}
  \to (q_m)_*\sO_{V_m}$ is split injective. $\triangle$\\*[2mm]
  \emph{Claim 2.} If $0\leq l<n-1$, then the natural maps $I_l/I_{l+1} \to 
  p_*p^*(I_l/I_{l+1})$ are split injective.\\*[1mm]
  \emph{Proof of Claim 2.} If $N\in \QCoh(T_m)$, then 
  $f_*(k_m)_*N=(i_m)_*(f_m)_*N$. Since $f_m$ is an affine morphism, it follows that $f_*(k_m)_* \colon \QCoh(T_m) \to \QCoh(X)$ is exact. If $P$ is one of the modules $K^{l+1}$, $K^{l+2}$, or 
  $K^{l+1}/K^{l+2}$, then $K^{m+1}P=0$, so the natural map $P \to (k_m)_*k_m^*P$ is an isomorphism. In particular, $f_*P \cong f_*(k_m)_*k_m^*P$. Hence, $I_l/I_{l+1} \cong f_*(k_m)_*k_m^*(K^{l+1}/K^{l+2})$ and the claim now follows from Claim 1. $\triangle$
  \\*[2mm]

  So we let $l\geq 0$ be an integer, which we assume to be $<n-1$. We will assume that $X_l$ has the resolution 
  property and we will now prove that $X_{l+1}$ has the resolution property. 
  Retaining the notation of Claim 1, there is a $2$-commutative
  diagram of algebraic stacks:
  \[
  \xymatrix@R=5mm{U_{l} \ar[dd]^{p_l} \ar[rr] & & U_{l+1} \ar[d] \\ & &
    \tilde{X}_{l+1} \ar[d] \\ X_l \ar[rr] \ar[urr] & & X_{l+1},  }
  \]
  where both the inner and outer squares are $2$-cartesian and the
  inner square is $2$-cocartesian. Let $Q_l = I_{l}/I_{l+1}$. The morphism $X_l \to
  \tilde{X}_{l+1}$ is a square zero extension of $X_l$ by
  $(p_l)_*p_l^*Q_l\cong p_*p^*Q_l$ and the morphism $\tilde{X}_{l+1} \to X_{l+1}$ is
  the morphism of $X_l$-extensions given by the natural map $Q_l \to
  (p_l)_*p_l^*Q_l$. By Claim 2, the morphism $Q_l \to (p_l)_*p_l^*Q_l$
  is split injective and so there is an induced splitting $X_{l+1}
  \to \tilde{X}_{l+1}$ which is affine. By
  \cite[Prop.~4.3(i)]{2013arXiv1306.5418G}, it remains to prove that
  $\tilde{X}_{l+1}$ has the resolution property, which is just
  Lemma \ref{L:res_strange}.
\end{proof}

\section{Formal gluings}\label{S:formal-gluings}
  Let $T$ be an algebraic stack, let $i \colon Z \inj T$ be a finitely presented closed
  immersion and let $j\colon U \to T$ denote its complement. A \emph{flat
    Mayer--Vietoris square} is a cartesian square of algebraic stacks
  \[
  \xymatrix{
    U'\ar[r]^{j'}\ar[d]_{\pi_U} & T'\ar[d]^\pi \\
    U\ar[r]^j & T\ar@{}[ul]|\square
  }
  \]
  such that $\pi$ is flat and $\pi|_Z$ is an
  isomorphism~\cite{MR1432058,mayer-vietoris}. If $F\colon \AlgSt^\op\to
  \Cat$ is a pseudo-functor, then there is a natural functor:
  \[
    \Phi_F\colon F(T)\to F(T')\times_{F(U')} F(U).
  \]
  Here $\AlgSt$ denotes the $2$-category of algebraic stacks. For the
  purposes of this paper, it is enough to consider pseudo-functors
  defined on affine schemes, that is, fibered categories over affine
  schemes. Indeed, our Mayer--Vietoris squares will be \emph{formal gluings}:
  $T=\Spec A$ is affine and noetherian, $Z=V(I)$, and
  $T'=\Spec \widehat{A}$, where $\widehat{A}$ is the $I$-adic completion
  of $A$.

  The following theorem follows from the main results of \cite{mayer-vietoris}
  (and almost from~\cite{MR1432058}).
  \begin{theorem}\label{T:formal-gluings}
    Consider a flat Mayer--Vietoris square as above.
    Let $X$ be an algebraic stack and consider the pseudo-functor
    $X_{\otimes}(-)=\Hom_{\otimes}(\QCoh(X),\QCoh(-))$ on the
    category of algebraic spaces.
    \begin{enumerate}
      \item $\Phi_{X_\otimes}$ is an equivalence of categories;
      \item $\Phi_X$ is fully faithful;
      \item $\Phi_X$ is an equivalence if $\Delta_X$ is quasi-affine;
      \item $\Phi_X$ is an equivalence if $X$ is Deligne--Mumford; and
      \item $\Phi_X$ is an equivalence if $T$ is locally excellent.
    \end{enumerate}
  \end{theorem}
  \begin{proof}
    By \cite[Thm.~B(1)]{mayer-vietoris} (or one of \cite[0.3]{MR1432058}
    and \cite[App.]{MR0272779} when $\pi$ is affine), there is
    an equivalence
    \[
      \QCoh(T)\to \QCoh(T')\times_{\QCoh(U')} \QCoh(U).
    \]
    Thus we have (i). Claims (ii) and (iii) are \cite[Thm.~B(3)]{mayer-vietoris} and claims 
    (iv) and (v) are \cite[Thm.~E and Thm.~A]{mayer-vietoris} respectively. Under some 
    additional assumptions: $\pi$ is affine, $\Delta_X$ is quasi-compact and separated,
    and in (v) $T'$ is locally noetherian; claims (ii)--(v) also follow from
    \cite[6.2 and 6.5.1]{MR1432058}.
  \end{proof}

  \begin{remark}\label{R:excellent}
    Recall that a noetherian ring $A$ is \emph{excellent}~\cite[p.~260]{MR1011461},
    \cite[Ch.~13]{MR575344} or~\cite[IV.7.8.2]{EGA}, if
    \begin{enumerate}
      \item $A$ is a G-ring, that is, $A_p\to \widehat{A_p}$ has geometrically
        regular fibers;
      \item the regular locus $\operatorname{Reg} B\subseteq \Spec B$ is open for every
        finitely generated $A$-algebra $B$; and
      \item $A$ is universally catenary.
    \end{enumerate}
    If (i) and (ii) hold, then we say that $A$ is \emph{quasi-excellent}.
    All excellency assumptions originate from~\cite{MR1432058,mayer-vietoris}
    via Theorem~\ref{T:formal-gluings}. The assumptions are
    used to guarantee that the formal fibers are geometrically regular so that
    N\'eron--Popescu desingularization applies. We can
    thus replace ``locally excellent'' with ``locally the spectrum of a
    G-ring''. Note that whereas being a G-ring and being quasi-excellent are
    local for the smooth topology~\cite[32.2]{MR1011461},
    excellency does not descend even for
    finite \'etale coverings~\cite[IV.18.7.7]{EGA}.
  \end{remark}

  \begin{corollary}\label{C:tannakian-over-closed+open}
    Let $X$ be an algebraic stack. Let $A$ be a ring and
    let $I\subset A$ be a finitely generated ideal. Let $T=\Spec A$, $Z=V(I)$ and
    $U=T\setminus Z$. Let $i \colon Z \to T$ and $j\colon U \to
    T$ be the resulting immersions.
    \begin{enumerate}
    \item Let $f_1$, $f_2 \colon T \to X$ be morphisms of algebraic
      stacks.
      \begin{enumerate}
      \item\label{CI:tannakian-f}
        Assume that $\ker(\sO_T \to j_*\sO_{U}) \cap \bigcap_{n=0}^\infty I^n=0$. Let $\alpha$, $\beta \colon
        f_1 \Rightarrow f_2$ be $2$-morphisms. If $\alpha_U = \beta_U$ and
        $\alpha_{Z^{[n]}} = \beta_{Z^{[n]}}$ for all $n$, then $\alpha=\beta$.
      \item\label{CI:tannakian-ff}
        Assume that $T$ is noetherian and that $\TF{X}(T)$ is faithful
        for all noetherian $T$. Let $t \colon f_1^* \Rightarrow f_2^*$ be
        a natural transformation of cocontinuous tensor functors. If
        $j^*(t)$ and $(i^{[n]})^*(t)$ are realizable for all $n$, then
        $t$ is realizable.
      \end{enumerate}
    \item\label{CI:tannakian-ess-surj}
      Assume either (a) $T$ is excellent, or (b) $T$ is noetherian and
      $X$ has quasi-affine or unramified diagonal.
      Further, assume that $\TFiso{X}(T)$ is fully
      faithful for all noetherian $T$. Let $f^*\colon \QCoh(X)
      \to\QCoh(T)$ be a cocontinuous tensor functor that \emph{preserves
        sheaves of finite type}. If $j^*f^*$
      and $(i^{[n]})^*f^*$ are algebraic for all $n$, then $f^*$ is
      algebraic.
    \end{enumerate}
  \end{corollary}
  The assumption in \ref{CI:tannakian-f} says that the
  filtration $\{\emptyset\inj Z\inj T\}$ is \emph{separating}
  (Definition~\ref{D:separating-filtration}). This is automatic if $T$ is
  noetherian (Lemma~\ref{L:separating}).

  \begin{proof}[Proof of Corollary~\ref{C:tannakian-over-closed+open}]
    First, we show \ref{CI:tannakian-ess-surj}. By assumption, the induced functor $(i^{[n]})^*f^*$ comes
    from a morphism $f^{[n]}\colon Z^{[n]}\to X$.
    Pick an \'etale cover $q\colon \tilde{Z}\to Z$
    such that $f^{[0]}\circ q\colon \tilde{Z}\to X$ has a lift $g\colon
    \tilde{Z}\to W$, where $p\colon W \to X$ is a smooth covering and $W$
    is affine. Descent (Lemma 
    \ref{L:fppf_local_tannakian}\ref{LI:fppf_local_tannakian:alg}) implies that we are free 
    to replace $T$ with an \'etale cover, so we
    may assume that $f$ also has a lift $g\colon
    Z\to W$ \cite[IV.18.1.1]{EGA}.

    Since $p$ is smooth, we may choose compatible lifts $g^{[n]}\colon Z^{[n]}\to W$ of
    $f^{[n]}$ for all $n$. But $W$ is affine, so there is an induced morphism
    $\hat{g}\colon \hat{T} \to W$, where $\hat{T}=\Spec \hat{A}$ and $\hat{A}$
    denotes the completion of $A$ at the ideal $I$. Let $\hat{f}=p\circ
    \hat{g}$. Then $(i^{[n]})^*\hat{f}^*=(f^{[n]})^*=(i^{[n]})^*f^*$ for all $n$. Since
    $\Coh(\hat{T})=\varprojlim_n \Coh(Z_n)$ (Lemma \ref{L:re_coh}), it follows that $\hat{f}^*\simeq \pi^*f^*$
    where $\pi\colon \hat{T}\to T$ is the completion morphism. Indeed, this last equivalence may be verified after restricting both sides to quasi-coherent $\sO_X$-modules of finite type (Example \ref{E:stk_fg_description}) and both sides send quasi-coherent $\sO_X$-modules of finite type to $\Coh(\hat{T})$. 

    Let $\hat{\jmath} \colon \hat{U} \to \hat{T}$ be the pullback of $j$ along $\pi$; then 
    we obtain a flat Mayer--Vietoris square:
    \[
  \xymatrix{
    \hat{U}'\ar[r]^{\hat{\jmath}}\ar[d]_{\pi_U} & \hat{T} \ar[d]^\pi \\
    U\ar[r]^j & T.\ar@{}[ul]|\square
  }
  \]
  Since $U$ and $\hat{U}$ are noetherian, $\TFiso{X}({U})$ and $\TFiso{X}(\hat{U})$ 
  are fully faithful. Thus, there is an essentially unique morphism of algebraic stacks $h 
  \colon U \to X$ such that $h^*\simeq j^*f^*$. But there are isomorphisms:
  \[
  \hat{\jmath}^*\hat{f}^*\simeq \hat{\jmath}^*\pi^*f^* \simeq \pi_U^*j^*f^* 
  \simeq \pi_U^*h^*,
  \]
  so $\hat{f}\circ \hat{\jmath} \simeq h\circ \pi_U$. 
  That $f^*$ is algebraic now follows from Theorem~\ref{T:formal-gluings}.

    For \ref{CI:tannakian-ff}, we proceed similarly. Consider the representable
    morphism $E \to T$ given by the equalizer of $f_1$ and $f_2$.
    Then $2$-isomorphisms between $f_1$ and $f_2$ correspond to
    $T$-sections of $E$. By assumption, we have compatible sections $\tau_U \in
    E(U)$ and $\tau^{[n]} \in E(Z^{[n]})$ for all $n$. Choose an \'etale
    presentation $E'\to E$ by an affine scheme $E'$. We may replace
    $T$ with an \'etale cover (Lemma~\ref{L:fppf_local_tannakian}\ref{LI:fppf_local_tannakian:ff}) and thus assume that $\tau^{[0]}$ lifts to
    $E'$.
    In particular, there are compatible lifts of all the
    $\tau^{[n]}$ to $E'$. Since $E'$ is affine, we get an induced morphism
    $\hat{T}\to E'$; thus, a morphism $\hat{T} \to E$. Equivalently, we get
    a $2$-isomorphism between $f_1\circ
    \pi$ and $f_2\circ \pi$. The induced $2$-isomorphism between
    $\pi^*f_1^*$ and $\pi^*f_2^*$ equals $\pi^*t$ since it coincides on
    the truncations. We may now apply
    Theorem~\ref{T:formal-gluings} to deduce that $t$ is realized by a
    $2$-morphism $\tau\colon f_1\Rightarrow f_2$.

    For \ref{CI:tannakian-f}, we consider the representable morphism $r\colon R \to T$ given
    by the equalizer of $\alpha$ and $\beta$. It suffices to prove that $r$ is
    an isomorphism. Note that $r$ is always a monomorphism and
    locally of finite presentation. By assumption, there are compatible
    sections of $r$ over $U$ and $Z^{[n]}$ for all $n$, thus $r_U$ and $r_{Z^{[n]}}$
    are isomorphisms for all $n$. By Proposition
    \ref{P:mono-with-stratification-is-iso}, $r$ is an isomorphism.
  \end{proof}
  \begin{remark}
    We do not know if the condition that $f^*$ preserves sheaves of finite type in
    \ref{CI:tannakian-ess-surj} is necessary. We do know that for any sheaf $F$
    of finite type, the restrictions of $f^*F$ to $U$ and $Z^{[n]}$ are coherent but
    this does not imply that $f^*F$ is coherent. For example, if $A=k[[x]]$,
    and $I=(x)$, then the $A$-module $k((x))/k[[x]]$ is not finitely generated
    but becomes $0$ after tensoring with $A/(x^n)$ or $A_x$.
  \end{remark}

\section{Tannaka duality}\label{S:filtrations}
In this section, we prove our general Tannaka duality result
(Theorem~\ref{T:main_tannaka_non-noeth}) and as a consequence 
also establish Theorem~\ref{T:main_tannaka:aff_stab}. To accomplish
this, we consider the following refinement of
\cite[Def.~2.5]{hallj_dary_alg_groups_classifying}.
\begin{definition}\label{D:fp-filtration}
Let $X$ be a quasi-compact algebraic stack. A \emph{finitely presented
  filtration of $X$} is a sequence of finitely presented closed immersions
$\emptyset=X_0\inj X_1\inj X_2\inj\dots \inj X_r\inj X$ such that
$|X_r|=|X|$. The \emph{strata} of the filtration are the locally closed
finitely presented substacks $Y_k:=X_k\setminus X_{k-1}$. The $n$th
\emph{infinitesimal neighborhood of $X_k$} is the finitely presented closed
immersion $X_k^{[n]}\inj X$ which is given by the ideal $I_k^{n+1}$ where
$X_k\inj X$ is given by $I_k$. The $n$th infinitesimal neighborhood of $Y_k$
is the locally closed finitely presented substack
$Y_k^{[n]}:=X_k^{[n]}\setminus X_{k-1}$.
\end{definition}

Stacks that have affine stabilizers can be stratified into stacks with
the resolution property.
\begin{proposition}\label{P:filtrations-exists}
Let $X$ be an algebraic stack. The following are equivalent:
\begin{enumerate}
\item $X$ is quasi-compact and quasi-separated with affine stabilizers;
\item $X$ has a finitely presented filtration $(X_k)$ with strata of
the form $Y_k=[U_k/\GL_{N_k}]$ where $U_k$ is quasi-affine.
\item $X$ has a finitely presented filtration $(X_k)$ with strata $Y_k$ that are
  quasi-compact with affine diagonal and the resolution property.
\end{enumerate}
\end{proposition}
\begin{proof}
That (i)$\implies$(ii)
is~\cite[Prop.~2.6(i)]{hallj_dary_alg_groups_classifying}.
That (ii)$\iff$(iii) is the Totaro--Gross theorem~\cite{2013arXiv1306.5418G}. That (iii)$\implies$(i)
is straightforward.
\end{proof}
When in addition $X$ is noetherian or, more generally, $X$ has finitely
presented inertia, this
result is due to Kresch~\cite[Prop.~3.5.9]{MR1719823} and
Drinfeld--Gaitsgory~\cite[Prop.~2.3.4]{MR3037900}. They construct
stratifications by quotient stacks of the form $[V_k/\GL_{N_k}]$, where each
$V_k$ is quasi-projective and the action is linear. This implies that the
strata have the resolution property. When $X$ has finitely presented
inertia the situation is simpler since $X$
can be stratified into gerbes~\cite[Cor.~8.4]{rydh-2014}, something which
is not possible in general.

\begin{remark}
In~\cite[Def.~1.1.7]{MR3037900}, Drinfeld and Gaitsgory introduces the notion
of a QCA stack. These are (derived) algebraic stacks that are quasi-compact and
quasi-separated with affine stabilizers and finitely presented inertia. The
condition on the inertia is presumably only used
for~\cite[Prop.~2.3.4]{MR3037900} and could be excised using
Proposition~\ref{P:filtrations-exists}.
\end{remark}

We now state and prove the main result of the paper. 
\begin{theorem}\label{T:main_tannaka_non-noeth}
  Let $T$ and $X$ be algebraic stacks and consider the functor
  \begin{align*}
  \TF{X}(T)\colon \Hom(T,X) &\to \Hom_{c\otimes}(\QCoh(X),\QCoh(T))
  \end{align*}
  and its variants $\TF{X}^{\ft}(T)$, $\TFiso{X}(T)$ and $\TFiso{X}^{\ft}(T)$ (see \S\ref{S:tensorial}). 
  Assume that $X$ is quasi-compact and quasi-separated.
  \begin{enumerate}
    \item\label{ti:tannaka:qg}
      If $X$ has quasi-affine diagonal, then
      \begin{enumerate}
        \item\label{ti:tannaka_ff:qg} $\TF{X}(T)$ is
          fully faithful; and
        \item\label{ti:tannaka_surj:qg} $\TF{X}^\ft(T)$ is essentially surjective if
          $T$ is locally noetherian.
      \end{enumerate}
    \item\label{ti:tannaka:wg}
      If $X$ has affine stabilizers, then
      \begin{enumerate}
      \item\label{ti:tannaka_faith:noemb/noeth}
        $\TF{X}(T)$ is faithful if $T$ is
        locally noetherian or has no embedded components;
      \item\label{ti:tannaka_full:core} $\TFiso{X}(T)$ is full if
        $T$ is locally noetherian;
      \item\label{ti:tannaka_full} $\TF{X}(T)$ is full if
        $X$ is affine-pointed and $T$ is locally noetherian.
      \item\label{ti:tannaka_surj:exc}
        $\TF{X}^\ft(T)$ is essentially surjective if $T$ is locally excellent, or
         $T$ is locally noetherian and $X$ is Deligne--Mumford.
      \end{enumerate}
  \end{enumerate}
  In particular, $\TFiso{X}^\ft(T)$ is an equivalence if $X$ has affine stabilizers
  and $T$ is locally excellent,
  and $\TF{X}^\ft(T)$ is an equivalence if $T$ is locally noetherian and $X$ either has quasi-affine diagonal or is Deligne--Mumford.
\end{theorem}
\begin{proof}
When $X$ has quasi-affine diagonal, we have already seen that $\TF{X}(T)$ is
fully faithful for all $T$ (Proposition~\ref{P:quaff_full_faithful}). This is \ref{ti:tannaka_ff:qg}.

Choose a filtration $(X_k)$ with strata $(Y_k)$ as in
Proposition~\ref{P:filtrations-exists}. We will prove the theorem by induction
on the number of strata $r$. If $r=0$, then $X=\emptyset$ and there is nothing
to prove.
If $r\geq 1$, then $U:=X\setminus X_1$ has a filtration of length $r-1$; thus
by induction the theorem holds for $U$. The theorem also holds for
$X_1^{[n]}=Y_1^{[n]}$ and all $n$, since $\TF{X_1^{[n]}}(T)$ is an
equivalence of categories by the Main Lemma~\ref{L:funny_nil_res}.
Note that if $r=1$, then $U=\emptyset$ and $X=X_1^{[n]}=Y_1^{[n]}$ for
sufficiently large $n$.

Let $I\subseteq \sO_X$ be the ideal
defining $Z=X_1$. Let $i^{[n]}\colon Z^{[n]}\inj X$ be the closed substack defined
by $I^{n+1}$ and let $j\colon U\to X$ be its complement.

For \ref{ti:tannaka_faith:noemb/noeth},
pick two maps $f_1,f_2\colon T\to X$ and $2$-isomorphisms
$\tau_1,\tau_2:f_1\Rightarrow f_2$ and assume that
$\TF{X}(T)(\tau_1)=\TF{X}(T)(\tau_2)$. We need to prove that
$\tau_1=\tau_2$.

For \ref{ti:tannaka_full:core} (resp.~\ref{ti:tannaka_full}), pick two maps $f_1,f_2\colon T\to X$ and a natural
isomorphism (resp.~transformation)
$\gamma\colon f_1^*\Rightarrow f_2^*$ of cocontinuous tensor functors. We need to prove
that $\gamma$ is realizable.

For \ref{ti:tannaka_surj:qg} and \ref{ti:tannaka_surj:exc}, pick a cocontinuous
tensor functor $f^*\colon \QCoh(X)\to \QCoh(T)$ preserving sheaves of finite type.
We need to prove that $f^*$ is algebraic.

When we prove \ref{ti:tannaka_surj:exc} (resp.\ \ref{ti:tannaka_full:core} and
\ref{ti:tannaka_full}), we assume that \ref{ti:tannaka_full:core}
(resp.\ \ref{ti:tannaka_faith:noemb/noeth}) already has been established. When
we prove \ref{ti:tannaka_surj:qg}, we note that $\omega_X(T)$ is fully faithful
for all $T$.
By Lemma \ref{L:fppf_local_tannakian}, it is enough to prove the results
when $T=\Spec A$ is affine.

In cases \ref{ti:tannaka_faith:noemb/noeth}, \ref{ti:tannaka_full:core} and \ref{ti:tannaka_full}, let $I_T = \Im(f_2^*I \to f_2^*\sO_X = \sO_T)$, which is a finitely generated ideal because $f_2$ is a morphism. In cases \ref{ti:tannaka_surj:qg} and \ref{ti:tannaka_surj:exc}, let
$I_T=\Im(f^*I\to f^*\sO_X=\sO_T)$, which is a finitely generated ideal because $T$ is noetherian. Let
$i_T^{[n]}\colon Z^{[n]}_T\inj T$ be the finitely presented closed immersion defined
by $I_T^{n+1}$ and let $j_T\colon U_T\inj T$ be its complement, a quasi-compact open
immersion.

In cases \ref{ti:tannaka_faith:noemb/noeth}, \ref{ti:tannaka_full:core} and \ref{ti:tannaka_full}, we have that
$U_T=f_1^{-1}(U)=f_2^{-1}(U)$; in the first case this is obvious and for the other two cases this follows from Corollary \ref{C:points-affine-stabs} and Lemma \ref{L:aff-pointed-use}, respectively.
We also have that
$Z^{[n]}_T=f_2^{-1}(Z^{[n]})\inj f_1^{-1}(Z^{[n]})$.
Thus, after restricting to either $Z^{[n]}_T$ or $U_T$ we have that $\tau_1=\tau_2$
in case \ref{ti:tannaka_faith:noemb/noeth} and that $\gamma$ is realizable in
cases \ref{ti:tannaka_full:core} and \ref{ti:tannaka_full}.

In cases \ref{ti:tannaka_surj:qg} and \ref{ti:tannaka_surj:exc}, Theorem \ref{T:tensor-loc-for-algebraic-stacks} produces for every $n\geq 0$ essentially unique cocontinuous tensor functors $f_U^* \colon \QCoh(U) \to \QCoh(U_T)$ and $f_{Z^{[n]}}^* \colon \QCoh(Z^{[n]}) \to \QCoh(Z_T^{[n]})$ such that $j_T^*f^*\simeq f_U^*j^*$ and $(i_T^{[n]})^*f^*\simeq (f_{Z^{[n]}})^*(i^{[n]})^*$. By the inductive assumption, $f_U^*$ is algebraic and the case $r=1$ implies that $f_{Z^{[n]}}^*$ is algebraic for each $n\geq 0$. In particular, $j_T^*f^*$ and $(i_T^{[n]})^*f^*$ is algebraic for each $n\geq 0$.

If $T$ is noetherian or has no embedded associated
points, then the stratification $\emptyset\subset Z_T\subset T$ is separating
by Lemma~\ref{L:separating}.
The result now follows from Corollary~\ref{C:tannakian-over-closed+open}.
\end{proof}
\begin{remark}\label{R:faithfulness-for-flat-over-noetherian}
Let $X$ be a quasi-compact and quasi-separated algebraic stack with affine
stabilizers. Let $T$ be a locally noetherian stack and
let $\pi\colon T'\to T$ be a flat morphism. Assume that we have morphisms
$f_1,f_2\colon T\to X$. Then
$\Hom(f_1\circ \pi,f_2\circ \pi)\to \Hom_\otimes(\pi^*f_1^*,\pi^*f_2^*)$ is
injective even if $T'$ is not noetherian.
Indeed, the stratification on $T'$ constructed in the proof of
Theorem \ref{T:main_tannaka_non-noeth}~\ref{ti:tannaka_faith:noemb/noeth},
is the pull-back along $\pi$ of a
stratification on $T$, hence separating by Lemma~\ref{L:separating}.
\end{remark}

We conclude with the proof of Theorem~\ref{T:main_tannaka:aff_stab}.
\begin{proof}[Proof of Theorem~\ref{T:main_tannaka:aff_stab}]
We note that
\[
\Hom_{r\otimes,\simeq}(\Coh(X),\Coh(T)) \to
\Hom_{c\otimes,\simeq}^\ft(\QCoh(X),\QCoh(T))
\]
is an equivalence of categories. It is thus enough to prove
that $\TFiso{X}^\ft(T)$
is
an equivalence of groupoids, which follows from
Theorem~\ref{T:main_tannaka_non-noeth}.
\end{proof}
\section{Applications}
In this section, we address the applications outlined in the introduction.
\begin{proof}[Proof of Corollary~\ref{C:application-fpqc-stack}]
Let $T'\to T$ be an fpqc covering with $T$ locally excellent and $T'$ locally noetherian.
Since $X$ is an fppf-stack, we may assume that $T$ and $T'$ are affine and
that $T'\to T$ is faithfully flat.
Let
$T''=T'\times_T T'$. Since $X$ has affine stabilizers, the functor
$\TFiso{X}(T)$ is an equivalence, the functor $\TFiso{X}(T')$ is fully faithful
and the functor $\TF{X}(T'')$ is faithful
for morphisms $T''\to T'\to X$ (Theorem~\ref{T:main_tannaka_non-noeth}
and Remark~\ref{R:faithfulness-for-flat-over-noetherian}).
Since $\Hom_{c\otimes,\simeq}(\QCoh(X),\QCoh(-))$ is an fpqc stack, it follows
that $T'\to T$ is a morphism of effective descent for $X$.
\end{proof}

\begin{proof}[Proof of Corollary~\ref{C:application-completions}]
It is readily verified that we can assume that $X$ is quasi-compact.
As $A$ is noetherian, $\Coh(A)=\varprojlim_n \Coh(A/I^n)$.
Thus,
\begin{align*}
X(A) &\iso \Hom_{r\otimes,\simeq}(\Coh(X),\Coh(A)) \\
&\iso \Hom_{r\otimes,\simeq}(\Coh(X),\varprojlim \Coh(A/I^n)) \\
&\iso \varprojlim \Hom_{r\otimes,\simeq}(\Coh(X),\Coh(A/I^n)) \\
&\iso \varprojlim X(A/I^n).\qedhere
\end{align*}
\end{proof}
\begin{proof}[Proof of Theorem~\ref{T:main_alg_hom}]
  First, we prove \ref{TI:main_alg_hom:rep}.
  We begin with the following standard reductions: we can assume that
  $S$ is affine; $X\to S$ is quasi-compact, so is of finite
  presentation; and $S$ is of finite type over $\Spec \Z$. 

  Since $S$ is now assumed to be excellent, we can prove the algebraicity of
  $\underline{\Hom}_S(Z,X)$ using a variant of Artin's criterion for
  algebraicity due to the first author
  \cite[Thm.~A]{hallj_openness_coh}. Hence, it is sufficient to prove that $\underline{\Hom}_S(Z,X)$ is
  \begin{enumerate}
    \renewcommand{\theenumi}{[\arabic{enumi}]}
    \renewcommand{\labelenumi}{\theenumi}
  \item\label{TPI:main_alg_hom:stk} a stack for the \'etale topology;
  \item\label{TPI:main_alg_hom:lp} limit preserving, equivalently, locally of finite presentation; 
  \item\label{TPI:main_alg_hom:homog} homogeneous, that is, satisfies a strong version of the
    Schlessinger--Rim criteria;
  \item\label{TPI:main_alg_hom:eff} effective, that is, formal deformations can be algebraized;
  \item\label{TPI:main_alg_hom:coh} the automorphisms, deformations, and obstruction functors are coherent.
  \end{enumerate}
  The main result of this article provides a method to prove \ref{TPI:main_alg_hom:eff} in 
  maximum generality, which we address first. Thus, let $T=\Spec B\to S$, where
$(B,\im)$ is a complete local noetherian ring. Let $T_n=\Spec(B/\im^{n+1})$. Since $Z \to S$ is proper, for every noetherian algebraic stack $W$ with affine stabilizers there are equivalences
\begin{align*}
\Hom(Z\times_S T,W) &\iso \Hom_{r\otimes,\simeq}(\Coh(W),\Coh(Z\times_S T)) & \mbox{(Theorem~\ref{T:main_tannaka:aff_stab})}\\
&\iso \Hom_{r\otimes,\simeq}(\Coh(W),\varprojlim \Coh(Z\times_S T_n))  & \mbox{\cite[Thm.~1.4]{MR2183251}}\\
&\iso \varprojlim \Hom_{r\otimes,\simeq}(\Coh(W),\Coh(Z\times_S T_n)) & \mbox{(Lemma \ref{L:re_coh})} \\
&\iso \varprojlim \Hom(Z\times_S T_n,W)  & \mbox{(Theorem~\ref{T:main_tannaka:aff_stab})}.
\end{align*}
Since $X$ and $S$ have affine stabilizers, it follows that
\[
\Hom_S(Z\times_S T, X) \iso \varprojlim \Hom_S(Z\times_S T_n,X);
\]
that is, the stack $\underline{\Hom}_S(Z,X)$ is effective and so satisfies \ref{TPI:main_alg_hom:eff}. 

The remainder of Artin's conditions are routine, so we will just sketch the arguments and provide pointers to the literature where they are addressed in more detail. Condition \ref{TPI:main_alg_hom:stk} is just \'etale descent and 
\ref{TPI:main_alg_hom:lp} is standard---see, for example, 
\cite[Prop.~4.18]{MR1771927}. For conditions \ref{TPI:main_alg_hom:homog} and \ref{TPI:main_alg_hom:coh}, it will be convenient 
to view $\underline{\Hom}_S(Z,X)$ as a substack of another moduli problem. This lets us  avoid having to directly discuss the deformation theory of non-representable morphisms of algebraic stacks.

If $W \to S$ is a morphism of algebraic stacks, let $\underline{\mathrm{Rep}}_{W/S}$ 
denote the $S$-groupoid that assigns to each $S$-scheme $T$ the category of 
\emph{representable} morphisms of algebraic stacks $V \to W\times_S T$ such that the 
composition $V \to W\times_S T \to T$ is proper, flat and of finite presentation. There is a 
morphism of $S$-groupoids: $\Gamma\colon \underline{\Hom}_S(Z,X) \to 
\underline{\mathrm{Rep}}_{Z\times_S X/S}$, which is given by sending a $T$-morphism 
$f\colon Z\times_S T \to X\times_S T$ to its graph $\Gamma(f) \colon Z\times_S T \to 
(Z\times_S X)\times_S T$. It is readily seen that $\Gamma$ is formally 
\'etale since $Z \to S$ is flat. Hence, it is sufficient to verify conditions \ref{TPI:main_alg_hom:homog} and 
\ref{TPI:main_alg_hom:coh} for $\underline{\mathrm{Rep}}_{Z\times_S X/S}$ 
\cite[Lemmas 1.5(9), 6.3 \& 6.11]{hallj_openness_coh}. That $\underline{\mathrm{Rep}}_{Z\times_S X/S}$ is homogeneous follows immediately from \cite[Lem.~9.3]{hallj_openness_coh}. A description of the automorphism, deformation and obstruction functors of $\underline{\mathrm{Rep}}_{Z\times_S X/S}$ in terms of the cotangent complex are given on \cite[p.\ 37]{hallj_openness_coh}, which mostly follows  from the results of \cite{MR2206635}. That these functors are coherent is \cite[Thm.~C]{hallj_coho_bc}. This completes the proof of \ref{TI:main_alg_hom:rep}.

We now address \ref{TI:main_alg_hom:qs} and \ref{T:main_alg_hom:diag}, that is, the separation properties of the algebraic stack $\underline{\Hom}_S(Z,X)$ relative to $S$. Let $T$ be an affine scheme. Let $Z_T$ and $X_T$ denote $Z\times_S T$ and $X\times_S T$, respectively. Suppose we are given two $T$-morphisms $f_1,f_2\colon Z_T\to
X_T$ and consider $Q:=\underline{\Isom}_{Z_T}(f_1,f_2)=X\times_{X\times_S X} Z_T$. Then $Q\to Z_T$ is
representable and of finite presentation. If $\pi \colon Z_T\to T$ denotes
the structure morphism, then $\pi_*Q$ is an algebraic space which is locally of finite presentation, being the pull-back of the diagonal of $\underline{\Hom}_S(Z,X)$
along the morphism $T\to \underline{\Hom}_S(Z,X)\times_S \underline{\Hom}_S(Z,X)$
corresponding to $(f_1,f_2)$.

Let $P$ be one of the properties: affine, quasi-affine, separated, quasi-separated. Assume 
that $\Delta_X$ has $P$; then $Q \to Z_T$ has $P$. We claim that the induced morphism 
$\pi_*Q\to T$ has $P$. For the properties affine and quasi-affine, this is \cite[Thm.~2.3 (i),(ii)]{hallj_dary_g_hilb_quot}. For quasi-separated (resp.~separated), this
is~\cite[Thm.~2.3 (ii),(iv)]{hallj_dary_g_hilb_quot} applied to the
quasi-affine morphism (resp.~closed immersion) $Q\to Q\times_Z Q$ and the Weil
restriction $\pi_*Q\to \pi_*Q\times_T \pi_*Q=\pi_*(Q\times_Z Q)$. In particular, we have proved that $\underline{\Hom}_S(Z,X)$ is algebraic and locally of finite presentation with quasi-separated diagonal over $S$. 

Now by Theorem \ref{T:rel-boundedness-affine-fibers}, $\Delta_{\underline{\Hom}_S(Z,X)/S}=\underline{\Hom}_S(Z,\Delta_{X/S})$ is of finite presentation, so $\underline{\Hom}_S(Z,X)$ is also quasi-separated. It remains to prove that it has affine stabilizers. To see this, we may assume that $T$ is the spectrum of an algebraically closed field. In this situation, either $\pi_*Q$ is empty or $f_1\simeq f_2$; it suffices to treat the latter case. In the latter case, $T \to X\times_S X$ factors through the diagonal $\Delta_{X/S} \colon X \to X\times_S X$, so it is sufficient to prove that $\underline{\Hom}_S(Z,I_{X/S})$, where $I_{X/S} \colon X\times_{X\times_S X} X \to X$ is the inertia stack, has affine fibers. But $I_{X/S}$ defines a group over $X$ with affine fibers, and the result follows from Theorem \ref{T:rel-boundedness-affine-fibers}.
\end{proof}
\begin{lemma}\label{L:weilr-hom}
Let $f\colon Z\to S$ be a proper and flat morphism of finite presentation
between algebraic stacks. For any morphism $X\to Z$ of algebraic stacks, the
forgetful morphism $f_*X \to \underline{\Hom}_S(Z,X)$ is an open immersion.
\end{lemma}
\begin{proof}
  It is sufficient to prove that if
  $T$ is an affine $S$-scheme and $h\colon Z\times_S T \to X\times_S
  T$ is a $T$-morphism, then the locus of points where $f_T\circ h\colon
  Z \times_S T \to Z\times_S T$ is an isomorphism is open on
  $T$.

  First, consider the diagonal of $f_T\circ h$. This morphism is
  proper and representable and the locus on $T$ where this map is a
  closed immersion is open~\cite[Lem.~1.8 (iii)]{MR2821738}.
  We may thus assume that $f_T\circ h$ is
  representable. Repeating the argument on $f_T\circ h$, we may assume that
  $f_T\circ h$ is a closed immersion. That the locus in $T$
  where $f_T\circ h$ is an isomorphism is open now follows easily by
  studying the \'etale locus of $f_T\circ h$, cf.\ \cite[Lem.~5.2]{MR2239345}.
  The result follows.
\end{proof}

\begin{proof}[Proof of Theorem~\ref{T:main_alg_weilr}]
  That $f_*X\to S$ is algebraic, locally of finite presentation, with
  quasi-compact and quasi-separated diagonal and affine stabilizers follows
  from Theorem~\ref{T:main_alg_hom} and Lemma~\ref{L:weilr-hom}. The
  additional separation properties of $f_*X$ follows
  from~\cite[Thm.~2.3 (i), (ii) \& (iv)]{hallj_dary_g_hilb_quot} applied
  to the diagonal and double diagonal of $X\to Z$.
\end{proof}

As claimed in the introduction, we now extend \cite[Thm.~2.3 \&
Cor.~2.4]{hallj_dary_g_hilb_quot}. The statement of the following
corollary uses the notion of a morphism of algebraic stacks that is
\emph{locally of approximation type} \cite[\S
1]{hallj_dary_g_hilb_quot}. A trivial example of a morphism
locally of approximation type is a quasi-separated morphism
that is locally of finite presentation. It is hoped that every
quasi-separated morphism of algebraic stacks is locally of
approximation type, but this is currently unknown. It is known,
however, that morphisms of algebraic stacks that have quasi-finite
and locally separated diagonal are locally of approximation
type~\cite{rydh-2009}. In particular, all quasi-separated morphisms of
algebraic stacks that are relatively Deligne--Mumford are locally of
approximation type.
\begin{corollary}\label{C:main_alg_loc-approx}
  Let $f\colon Z \to S$ be a proper and flat morphism of finite
  presentation between algebraic stacks.
  \begin{enumerate}
  \item\label{CI:loc-approx-type}
    Let $h\colon X \to S$ be a morphism of algebraic stacks with
    affine stabilizers that is locally of approximation type. Then
    $\underline{\Hom}_S(Z,X)$ is algebraic and locally of approximation type
    with affine stabilizers. If $h$ is locally of finite presentation, then so
    is $\underline{\Hom}_S(Z,X)\to S$. If the diagonal of $h$ is affine
    (resp.~quasi-affine, resp.~separated), then so is the diagonal
    of $\underline{\Hom}_S(Z,X)\to S$.
  \item\label{CI:weilr}
    Let $g\colon X \to Z$ be a morphism of algebraic stacks
    such that $f\circ g: X \to S$ has affine stabilizers and is
    locally of approximation type. Then
    the $S$-stack $f_*X$ is algebraic and locally of approximation type
    with affine stabilizers. If $g$ is locally of finite presentation, then so
    is $f_*X\to S$. If the diagonal of $g$ is affine (resp.~quasi-affine,
    resp.~separated), then so is the diagonal of $f_*X\to S$.
  \end{enumerate}
\end{corollary}
\begin{proof}
  For \ref{CI:loc-approx-type}, we may
  immediately reduce to the situation where $S$ is an affine
  scheme. Since $f$ is quasi-compact, we may further assume that $h$
  is quasi-compact. By \cite[Lem.~1.1]{hallj_dary_g_hilb_quot}, there
  is an fppf covering $\{S_i \to S\}$ such that each $S_i$ is affine
  and $X\times_S S_i \to S_i$ factors as $X\times_S S_i \to X^0_i \to
  S_i$, where $X^0_i \to S_i$ is of finite presentation and $X\times_S
  S_i \to X_i^0$ is affine. Combining the results of
  \cite[Thm.~2.8]{hallj_dary_alg_groups_classifying} with
  \cite[Thms.~D \&~7.10]{rydh-2009}, we can arrange so that each $X_i^0 \to
  S$ has affine stabilizers (or has one of the other desired
  separation properties).

  Thus, we may now replace $S$ by $S_i$ and
  may assume that $X\to S$ factors as $X \xrightarrow{q} X_0 \to S$,
  where $q$ is affine and $X_0 \to S$ is of finite presentation with
  the appropriate separation condition. By Theorem
  \ref{T:main_alg_hom}, the stack $\underline{\Hom}_S(Z,X_0)$ is algebraic and locally of finite presentation with
  the appropriate separation condition.  By
  \cite[Thm.~2.3(i)]{hallj_dary_g_hilb_quot}, the morphism
  $\underline{\Hom}_S(Z,X) \to \underline{\Hom}_S(Z,X_0)$ is
  representable by affine morphisms; the result follows.

  For~\ref{CI:weilr} we argue exactly as in the proof of
  Theorem~\ref{T:main_alg_weilr}.
\end{proof}

\section{Counterexamples}
In this section we give four counter-examples (Theorems~\ref{T:failure_tensoriality}, \ref{T:failure_full}, \ref{T:failure_hom-stack}, and \ref{T:failure_completions}):
\begin{itemize}
\item in Theorems~\ref{T:main_tannaka:aff_stab},
  \ref{T:main_alg_hom} and
  \ref{T:main_tannaka_non-noeth}\ref{ti:tannaka_faith:noemb/noeth}
  it is necessary that $X$ has affine stabilizer
  groups;
\item in Theorem~\ref{T:main_tannaka_non-noeth}\ref{ti:tannaka_full},
  it is necessary that $X$ is affine-pointed;
\item in Theorem~\ref{T:main_alg_hom}, it is necessary that $X$ has affine
  stabilizer groups; and
\item in Corollary~\ref{C:application-completions}, it is necessary that $X$
  has affine stabilizer groups.
\end{itemize}

\begin{theorem}\label{T:failure_tensoriality}
  Let $X$ be a quasi-separated algebraic stack. If $k$ is an algebraically
  closed field and $x\colon \Spec k\to X$ is a point with non-affine
  stabilizer, then $\Aut(x)\to \Aut_{\otimes}(x^*)$ is not injective.
  In particular, $\TF{X}(\Spec k)$ is not faithful and $X$ is not
  tensorial.
\end{theorem}
\begin{proof}
  By assumption, the stabilizer group scheme $G_x$ of $x$ is not affine.
  Let $H=(G_x)_{\ant}$ be the largest anti-affine subgroup of $G_x$;
  then $H$ is a non-trivial anti-affine group scheme over $k$
  and the quotient group scheme $G_x/H$ is
  affine~\cite[\S III.3.8]{MR0302656}. The induced morphism
  $B_kH \to B_kG_x\to X$ is thus quasi-affine by~\cite[Thm.~B.2]{MR2774654}.

  By
  \cite[Lem.~1.1]{MR2488561}, the morphism $p\colon \Spec k
  \to B_kH$ induces an equivalence of abelian tensor 
  categories $p^*\colon \QCoh(B_kH) \to \QCoh(\Spec k)$. Since
  $\Aut(p)=H(k)\neq\{\id{p}\}=\Aut_\otimes(p^*)$, the functor
  $\TF{B_kH}(\Spec k)$ is
  not faithful. Hence $\TF{X}(\Spec k)$ is not faithful by
  Lemma~\ref{L:qaff-tensorial:rel}.
\end{proof}
We also have the following theorem.
\begin{theorem}\label{T:failure_full}
  Let $X$ be a quasi-compact and quasi-separated algebraic stack with
  affine stabilizers. If $k$ is a field and
  $x_0\colon \Spec k \to X$ is a non-affine morphism, then there exists
  a field extension $K/k$ and a point $y\colon \Spec K\to X$ such that
  $\Isom(y,x) \to \Hom_{\otimes}(y^*,x^*)$ is not surjective, where $x$
  denotes the $K$-point corresponding to $x_0$. In particular,
  $\TF{X}(\Spec K)$ is not full.
\end{theorem}
\begin{proof}
  To simplify notation, we let $x=x_0$.
  Since $X$ has quasi-compact diagonal, $x$ is
  quasi-affine~\cite[Thm.~B.2]{MR2774654}. By Lemma
  \ref{L:qaff-tensorial:rel}, we may replace $X$ by $\sSpec_X
  (x_*k)$ and consequently assume that $x$ is a
  quasi-compact open immersion and $\sO_X \to x_*k$ is an
  isomorphism. In particular, $x$ is a section to a morphism $f\colon
  X \to \Spec k$. Since $x$ is not affine, it follows that there
  exists a closed point $y$ disjoint from the image of $x$. In particular,
  there is a field extension $K/k$ and a $k$-morphism $y\colon \Spec K \to X$
  whose image is a closed point disjoint from $x$.
  
  We now base change the entire situation by $\Spec K \to \Spec
  k$. This results in two morphisms $x_K$, $y_K\colon \Spec K \to
  X\otimes_k K$, where $x_K$ is a quasi-compact open immersion such
  that $\sO_{X\otimes_k K} \cong (x_K)_*K$ and $y_K$ has image a
  closed point disjoint from the image of $x_K$. We replace
  $X$, $k$, $x$, and $y$ by $X\otimes_kK$, $K$, $x_K$, and $y_K$ respectively.

  Let $\mathcal{G}_y \subseteq X$ be the residual gerbe associated to
  $y$, which is a closed immersion. We define a natural transformation
  $\gamma^\vee \colon x_* \Rightarrow y_*$ at $k$ to be
  the composition $x_*k \cong \sO_X \surj \sO_{\mathcal{G}_y} \to
  y_*k$ and extend to all of $\QCoh(\Spec k)$ by taking
  colimits. By adjunction, there is an induced natural transformation
  $\gamma\colon y^* \to x^*$. A simple calculation shows that $\gamma$
  is a natural transformation of cocontinuous tensor functors. Since
  its adjoint $\gamma^\vee$ is not an isomorphism, $\gamma$ is not an
  isomorphism; thus $\gamma$ is not realizable. The result follows.
\end{proof}
The following lemma is a variant of \cite[Ex.~4.12]{2014arXiv1404.7483B}, which B.~Bhatt communicated to the authors.
\begin{lemma}\label{L:degeneration-of-nodal-curve}
Let $k$ be an algebraically closed field and let $G/k$ be an anti-affine group
scheme of finite type. Let $Z/k$ be a regular scheme with a closed
subscheme $C$ that is a nodal curve over $k$. Then there is a compatible system
of $G$-torsors $E_n\to C^{[n]}$ such that there does not exist a $G$-torsor
$E\to Z$ that restricts to the $E_n$s.
\end{lemma}
\begin{proof}
Recall that $G$ is smooth, connected and commutative \cite[\S III.3.8]{MR0302656}. Furthermore, by
Chevalley's theorem, there is an extension $0\to H\to G \to A\to 0$, where
$A$ is an abelian variety (of positive dimension) and $H$ is affine. Let
$x_A\in A(k)$ be an element of infinite order and let $x\in G(k)$
be any lift of $x_A$.

Let $\widetilde{C}$ be the normalization of $C$. Let $F_0\to C$ be the
$G$-torsor obtained by gluing the trivial $G$-torsor on
$\widetilde{C}$ along the node by translation by $x$. Note that the induced
$A$-torsor $F_0/H\to C$ is not torsion as it is obtained by gluing along the
non-torsion element $x_A$.

We may now lift $F_0\to C$ to $G$-torsors $F_n\to C^{[n]}$. Indeed, the
obstruction to lifting $F_{n-1}$ to $F_n$ lies in $\Ext^1_{\sO_C}(\LDERF
g_0^*L^\bullet_{BG/k},I^{n}/I^{n+1})$, where $g_0\colon C\to BG$
is the morphism corresponding to $F_0\to C$ and $I$ is the ideal defining $C$
in $Z$. Since $G$ is smooth, the cotangent complex
$L^\bullet_{BG/k}$ is concentrated in degree $1$ and since $C$ is a
curve, it has cohomological dimension $1$. It follows that the obstruction
group is zero.

Now given a $G$-torsor $F\to Z$, there is an induced $A$-torsor $F/H\to
Z$. Since $Z$ is regular, the torsor $F/H\to Z$ is torsion in
$H^1(Z,A)$~\cite[XIII 2.4 \& 2.6]{MR0260758}. Thus, $F/H\to Z$ cannot restrict
to $F_0/H\to C$ and the result follows.
\end{proof}
We now have the following theorem, which is a counterexample to \cite[Thm.~1.1]{MR2194377} and \cite[Case I]{MR2258535}. 
\begin{theorem}\label{T:failure_hom-stack}
Let $X \to S$ be a quasi-separated morphism of algebraic stacks. If $k$ is an algebraically closed field and $x\colon \Spec k \to X$ is a point with non-affine stabilizer, then there exists a morphism $\Aff^1_k \to S$ and a proper and flat family of curves $Z \to \Aff^1_k$, where $Z$ is regular, such that $\underline{\Hom}_{\Aff^1_k}(Z,X\times_S \Aff^1_k)$ is not algebraic. 
\end{theorem}
\begin{proof}
  Let $Q$ be the stabilizer group scheme of $x$ and let $G$
  be the largest anti-affine subgroup scheme of $Q$; thus, $G$ is a
  non-trivial anti-affine group scheme over $k$ and the quotient group
  scheme $Q/G$ is affine \cite[\S III.3.8]{MR0302656}.

  Let $Z$ be a proper family of curves over $T=\Aff^1_k=\Spec k[t]$ with
  regular total space and a nodal curve $C$ as the fiber over the origin; for example, take $Z=\Proj_{T}(k[t][x,y,z]/(y^2z-x^2z-x^3-tz^3))$ over $T$. Let $T_n=V(t^{n+1})$, $\hat{T} = \Spec \hat{\sO}_{T,0}$, $Z_n=Z\times_T T_n$, and $\hat{Z} = Z\times_T \hat{T}$. We now apply Lemma~\ref{L:degeneration-of-nodal-curve} to $C$ in $\hat{Z}$ and $G$. Since $Z_n = C^{[n]}$, this produces an element in
\[
\varprojlim_n \underline{\Hom}_T(Z,BG_T)(T_n)=\varprojlim_n \Hom(Z_n,BG)
\]
that does not lift to
\[
\underline{\Hom}_T(Z,BG_T)(\hat{T})=\Hom(\hat{Z},BG).
\]
This shows that $\underline{\Hom}_T(Z,BG_T)$ is not algebraic.

By \cite[Thm.~B.2]{MR2774654}, the morphism $x$ factors as $\Spec k
\to BQ\to \mathcal{Q} \to X$, where $\mathcal{Q}$ is the residual gerbe,
$\mathcal{Q} \to X$ is quasi-affine and $BQ \to \mathcal{Q}$ is affine.
Since $Q/G$ is affine, it follows that
the induced morphism $BG \to BQ\to X$ is quasi-affine. By
\cite[Thm.~2.3(ii)]{hallj_dary_g_hilb_quot}, the induced morphism
$\underline{\Hom}_T(Z,BG_T) \to \underline{\Hom}_T(Z,X\times_S T)$ is
quasi-affine. In particular, if $\underline{\Hom}_T(Z,X\times_S T)$ is
algebraic, then $\underline{\Hom}_T(Z,BG_T)$ is algebraic,
which is a contradiction. The result follows.
\end{proof}
The following theorem extends \cite[Ex.~4.12]{2014arXiv1404.7483B}.
\begin{theorem}\label{T:failure_completions}
Let $X$ be an algebraic stack with quasi-compact diagonal. If $X$ does not have
affine stabilizers, then there exists a noetherian two-dimensional regular ring
$A$, complete with respect to an ideal $I$, such that $X(A)\to \varprojlim
X(A/I^n)$ is not an equivalence of categories.
\end{theorem}
\begin{proof}
Let $x\in |X|$ be a point with non-affine stabilizer group. Arguing as in the proof of Theorem \ref{T:failure_hom-stack}, there exists an algebraically closed
field $k$, an anti-affine group scheme $G/k$ of finite type and a quasi-affine morphism $BG\to X$. An easy calculation shows that it is enough to prove the theorem for $X=BG$.

Let $A_0=k[x,y]$ and let $A$ be the completion of $A_0$ along the ideal
$I=(y^2-x^3-x^2)$. Then $Z=\Spec A$ and $C=\Spec A/I$ satisfies the conditions
of Lemma~\ref{L:degeneration-of-nodal-curve} and we obtain an element in
$\varprojlim_n X(A/I^n)$ that does not lift to $X(A)$.
\end{proof}

\appendix

\section{Monomorphisms and stratifications}
In this appendix, we introduce some notions and results needed for the
faithfulness part of Theorem~\ref{T:main_tannaka_non-noeth} when $T$ is not
noetherian. This is essential for the proof of Corollary \ref{C:application-fpqc-stack}. 

\begin{definition}\label{D:separating-filtration}
We say that a finitely presented filtration $(X_k)$ of $X$ (Definition~\ref{D:fp-filtration}) is \emph{separating}
if the family $\{j_k^n\colon Y_k^{[n]}\to X\}_{k,n}$ is
separating~\cite[IV.11.9.1]{EGA}, that is, if the intersection
$\bigcap_{k,n}\ker\bigl(\sO_X\to (j_k^n)_*\sO_{Y_k^{[n]}}\bigr)$ is zero as a
lisse-\'etale sheaf.
\end{definition}

\begin{lemma}\label{L:separating}
Every finitely presented filtration $(X_k)$ on $X$ is separating if either
\begin{enumerate}
\item $X$ is noetherian; or
\item $X$ has no embedded (weakly) associated point.
\end{enumerate}
If $X$ is noetherian with a filtration $(X_k)$ and $X'\to X$ is flat, then
$(X_k\times_X X')$ is a separating filtration on $X'$.
\end{lemma}
\begin{proof}
As the question is smooth-local, we can assume that $X$ and $X'$ are affine
schemes. If $X$ is noetherian, then by primary decomposition there exists a
separating family $\coprod_{i=1}^m \Spec A_i\to X$ where the $A_i$ are
artinian. As every $\Spec A_i$ factors through some $Y_k^{[n]}$, it
follows that $(X_k)$ is separating. In general, $\{\Spec \sO_{X,x}\to
X\}_{x\in \Ass(X)}$ is separating~\cite[1.2, 1.5, 1.6]{MR0168625}. If $x$
is a non-embedded associated point,
then $\Spec \sO_{X,x}$ is a one-point scheme and factors through some
$Y_k^{[n]}$ and the first claim follows.

For the last claim, we note that a finite number of the infinitesimal
neighborhoods of the strata suffices in the noetherian case and that
flat morphisms preserve kernels and finite intersections.
\end{proof}

\begin{proposition}\label{P:mono-with-stratification-is-iso}
Let $X$ be an algebraic stack with a finitely presented filtration $(X_k)$.
Let $f\colon Z\to X$ be a morphism locally of finite type. If
$f|_{Y_k^{[n]}}$ is an isomorphism for every $k$ and $n$, then $f$ is a
surjective closed immersion. If in addition $(X_k)$ is separating, then $f$ is
an isomorphism.
\end{proposition}
\begin{proof}
Note that $f$ is a surjective and quasi-compact monomorphism.
We will prove that $f$ is
a closed immersion by induction on the number of strata~$r$. If $r=0$, then
$X=\emptyset$ and there is nothing to prove. If $r=1$, then
$X=X_1^{[n]}=Y_1^{[n]}$ for sufficiently large $n$ and the result follows. If
$r\geq 2$, then let $U=X\setminus X_1$. By the induction hypothesis, $f|_U$ is
a surjective closed immersion.

It is enough to show that $f$ is a closed immersion in a neighborhood of every
$x\in |X_1|$. This can be checked locally in the \'etale topology and we may
thus assume that $Z=Z_0\amalg Z_1$ where $Z_0\to X$ is a closed immersion and
$Z_1\cap f^{-1}(x)=\emptyset$. Note that $f(Z_0)\cap U$ and $f(Z_1)\cap U$
are disjoint closed and open subsets.

It is further enough to show that $f$ is a closed immersion after replacing $X$
with either $X_1$, $\overline{f(Z_0)\cap U}$ or $\overline{f(Z_1)\cap U}$.
In the first and second
case, $f$ is certainly a closed immersion. In the third case, $f(Z_0)$ is
set-theoretically contained in $X_1$. Let $W=f(Z_0)\cap X_1$; this is an
open and closed substack of $X_1$. Thus, $f|_{Z_0}\colon Z_0\to X$ factors
through $W^{[n]}$ for sufficiently large $n$. By hypothesis, this means that
$Z_0\cong W^{[n]}\cong W^{[N]}$ for all sufficiently large $n$ and all $N\geq
n$. This implies that $W^{[n]}\inj X$ is an open immersion and we have
proved that $f$ is a closed immersion in a neighborhood of $x$. The result
follows.

The last claim is obvious.
\end{proof}

The following example illustrates that a closed immersion $f\colon Z\to X$ as in
Proposition~\ref{P:mono-with-stratification-is-iso} need not be an isomorphism
even if $f$ is of finite presentation.
\begin{example}
Let $A = k[x,z_1,z_2,\dots]/(xz_1,\{z_k-xz_{k+1}\}_{k\geq 1},\{z_iz_j\}_{i,j\geq 1})$ and $B = A/(z_1)$.
Then $A/(x^n) = k[x]/(x^n) = B/(x^n)$ and $A_x = k[x]_x = B_x$ but the surjection
$A\to B$ is not an isomorphism.
\end{example}

\newcommand{\norm}{\mathrm{norm}}
\newcommand{\cms}{\mathrm{cms}}
\section{A relative boundedness result for Hom~stacks}\label{S:boundedness}
Here we prove the following relative boundedness result for Hom stacks.

\begin{theorem}\label{T:rel-boundedness-affine-fibers}
Let $f\colon Z\to S$ be a proper, flat and finitely presented morphism of
algebraic stacks. Let $X$ and $Y$ be algebraic stacks that are locally of
finite presentation and quasi-separated over $S$ and have affine
stabilizers over $S$. Let $g\colon X\to Y$ be a finitely presented $S$-morphism.
If $g$ has affine fibers, then
\[
\underline{\Hom}_S(Z,g) \colon \underline{\Hom}_S(Z,X) \to \underline{\Hom}_S(Z,Y)
\]
is of finite presentation. If in addition $g\colon X \to Y$ is a group, then $\underline{\Hom}_S(Z,g)$ is a group with affine fibers. 
\end{theorem}
Theorem~\ref{T:rel-boundedness-affine-fibers} is used in Theorems \ref{T:main_alg_hom} and \ref{T:main_alg_weilr} to establish the
quasi-compactness of the diagonal of Hom-stacks and Weil
restrictions. 

Without using Theorem~\ref{T:rel-boundedness-affine-fibers},
the proof of Theorem \ref{T:main_alg_hom} give the algebraicity of the Hom-stacks and that they have quasi-separated diagonals. In the setting of 
Theorem~\ref{T:rel-boundedness-affine-fibers}, we may conclude that $\underline{\Hom}_S(Z,g)$ is a quasi-separated morphism of algebraic stacks that are locally of finite presentation over $S$. It remains to prove that the morphism $\underline{\Hom}_S(Z,g)$ is quasi-compact.

\subsection*{Preliminary reductions}
If $W$ and $T$ are algebraic stacks over $S$, let $W_T = W\times_S T$; similarly for morphisms between stacks over $S$. We will use this notation throughout this appendix. 

As the question is local on $S$, we may assume that $S$ is an affine scheme.
We may also assume that $X$ and $Y$ are of finite presentation over $S$ since
it is enough to prove the theorem after replacing $Y$ with an open
quasi-compact substack and $X$ with its inverse. By
standard approximation results, we may then assume that $S$ is of finite type over
$\Spec \Z$. For the remainder of this article, all stacks will be of finite
presentation over $S$ and hence excellent with finite normalization.

By noetherian induction on $S$, to prove that $\underline{\Hom}_S(Z,g)$ is
quasi-compact, we may assume that $S$ is integral and replace $S$ with a
suitable dense open subscheme. Moreover, we may also replace $Z\to S$ with the
pull-back along a dominant map $S'\to S$. Recall that there exists a field
extension $K'/K(S)$ such that $(Z_{K'})_\red$ (resp.\ $(Z_{K'})_\norm$) is
geometrically reduced (resp.\ geometrically normal) over $K'$. After replacing
$S$ with a dense open subset of the normalization in $K'$, we may thus assume
that
\begin{enumerate}
\item $Z_\red\to S$ is flat with geometrically reduced
  fibers; and
\item $Z_\norm\to S$ is flat with geometrically normal
  fibers;
\end{enumerate}
since these properties are
constructible~\cite[IV.9.7.7~(iii) and 9.9.4~(iii)]{EGA}.

We now prove three reduction results. Throughout, we will assume the following:
\begin{itemize}
\item $Z$ is proper and flat over $S$,
\item $X$ and $Y$ are finitely presented algebraic stacks over $S$ with affine stabilizers, and
\item $g \colon X \to Y$ is a representable morphism over $S$.
\end{itemize}
Our first reduction result is similar to~\cite[Lem.~5.11]{MR2239345}.
\begin{lemma}\label{L:QC-descent:nil}
If $\underline{\Hom}_{S'}(Z',g_{S'})$ is quasi-compact
for every scheme $S'$, morphism $S'\to S$ and nil-immersion $Z'\to Z_{S'}$ such that
$Z'\to S'$ is proper and flat with geometrically reduced fibers, then
$\underline{\Hom}_S(Z,g)$ is quasi-compact.
\end{lemma}
\begin{proof}
Assume that the condition holds. To prove that $\underline{\Hom}_S(Z,g)$ is quasi-compact,
we may assume that $S$ is integral. We may also assume that $Z_\red\to S$ has
geometrically reduced fibers. Pick a sequence of square-zero nil-immersions
$Z_\red=Z_0\inj Z_1\inj \dots \inj Z_n=Z$. After replacing $S$ with a dense
open subset, we may assume that all the $Z_i\to S$ are flat. Thus, it suffices to show that
if $j \colon Z_0 \to Z$ is a square-zero closed immersion where $Z_0$ is flat over $S$ and $\underline{\Hom}_S(Z_0,g)$ is quasi-compact, then $\underline{\Hom}_S(Z,g)$ is quasi-compact. Now argue as in \cite[Lem.~5.11]{MR2239345}, but this time using the deformation theory of \cite[Thm.~1.5]{MR2206635} and the Semicontinuity Theorem of \cite[Thm.~A]{hallj_coho_bc}.
\end{proof}

Before we proceed, we make the following observation: fix an $S$-scheme $T$ and an $S$-morphism $y\colon Z_T \to Y$. This corresponds to a map
$T\to \underline{\Hom}_S(Z,Y)$.
The pullback of $\underline{\Hom}_S(Z,g)$ along this map is isomorphic to the Weil 
restriction $\weilr_{Z_T/T}(X\times_{g,Y,y} Z_T)$, which we will denote as 
$H_{Z/S,g}(y)$. Note that our hypotheses guarantee that $H_{Z/S,g}(y)$ is locally of finite type 
and quasi-separated over $T$.

The second reduction is for a (partial) normalization.

\begin{lemma}\label{L:QC-descent:normalization}
If $\underline{\Hom}_{S'}(Z',g_{S'})$ is quasi-compact for every scheme $S'$, morphism
$S'\to S$ and finite morphism $Z'\to Z_{S'}$ such that $Z'\to S'$ is proper and flat
with geometrically normal fibers, then $\underline{\Hom}_S(Z,g)$ is quasi-compact.
\end{lemma}
\begin{proof}
By Lemma~\ref{L:QC-descent:nil}, we may assume that $Z\to S$ is flat
with geometrically reduced fibers.  We will use induction on the maximal fiber
dimension~$d$ of $Z\to S$. After modifying $S$, we may assume that
$W:=Z_\norm\to S$ is flat with geometrically normal fibers. Let $Z_0\inj Z$ and
$W_0\inj W$ be the closed substacks given by the conductor ideal of $W\to
Z$.

After replacing $S$ with a dense open subset, we may assume that $Z_0\to S$ and
$W_0\to S$ are flat and that $W\to Z$ is an isomorphism over an open subset
$U\subseteq Z$ that is dense in every fiber. In particular, since $Z_0\cap
U=\emptyset$, the dimensions of the fibers of $Z_0\to S$ are strictly smaller
than $d$. Thus, by induction we may assume that $\underline{\Hom}_S(Z_0,g)$ is
quasi-compact. But 
\[
\xymatrix{%
W_0\ar@{^(->}[r]^i\ar[d]_{h_0} & W\ar[d]^h\\
Z_0\ar@{^(->}[r]^j & Z\ar@{}[ul]|\square
}%
\]
is a bicartesian square and remains so after arbitrary base change over $S$
since $W_0\to S$ is flat. Indeed, that it is cartesian is 
\cite[Lem.~A.3(i)]{hallj_openness_coh}. That it is cocartesian and the commutes with arbitrary base change 
over $S$ follows from the 
arguments of \cite[Lem.~A.4, A.8]{hallj_openness_coh} and the existence of pinchings of algebraic spaces 
\cite[Thm.~38]{kollar-2008}. 

It remains to prove that $H_{Z/S,g}(y) \to T$ is quasi-compact, where $T$ is an integral scheme of finite type over $S$ and $y\colon Z_T \to Y$ is a morphism. The 
bicartesian square above implies that 
\[
H_{Z/S,g}(y) \simeq H_{Z_0/S,g}(y j) \times_{H_{W_0/S,g}(y  h  i)} H_{W/S,g}(y h).
\]
The result follows, since $H_{Z_0/S,g}(y j)$ and $H_{W/S,g}(y h)$ are quasi-compact and 
$H_{W_0/S,g}(y h i)$ is quasi-separated.

\end{proof}
We have the following variant of $h$-descent~\cite[Thm.~7.4]{MR2679038}.
\begin{lemma}\label{L:h-descent}
Let $S$ be an algebraic stack, let $T$ be an algebraic $S$-stack and
let $g\colon T'\to T$ be a universally
subtrusive (e.g., proper and surjective)
morphism of finite presentation such that $g$ is flat over an open substack
$U\subseteq T$. If $T$ is weakly normal in $U$ (e.g., $T$ normal and $U$ open
dense), then for every representable morphism $X\to S$, the following sequence
of sets is exact:
\[
\xymatrix{
X(T)\ar[r] & X(T')\ar@<.5ex>@{->}[r]\ar@<-.5ex>@{->}[r] & X(T'\times_T T')
}
\]
where $X(T)=\Hom_S(T,X)$ etc.
\end{lemma}
\begin{proof}
It is enough to prove that given a morphism $f\colon T'\to X$ such that
$f\circ\pi_1=f\circ\pi_2\colon T'\times_T T'\to X$,
there exists a unique morphism $h\colon T\to X$ such
that $f=h\circ g$. By fppf-descent over $U$, there is a unique $h|_U\colon U\to
X$ such that $f|_{g^{-1}(U)}=h|_U\circ g|_{g^{-1}(U)}$. Consider the morphism
$\tilde{g}\colon \widetilde{T'}=T'\amalg U\to T$. The morphism
$\tilde{f}=(f,h|_U)\colon \widetilde{T'}\to X$ satisfies
$\tilde{f}\circ\tilde{\pi}_1=\tilde{f}\circ\tilde{\pi}_2$ where $\tilde{\pi}_i$
denotes the projections of $\widetilde{T'}\times_X \widetilde{T'}\to \widetilde{T'}$. By assumption,
$\tilde{g}$ is universally subtrusive and weakly normal. Thus,
by $h$-descent~\cite[Thm.~7.4]{MR2679038}, we have an exact sequence
\[
\xymatrix{
X(T)\ar[r] & X(\widetilde{T'})\ar@<.5ex>@{->}[r]\ar@<-.5ex>@{->}[r] & X\bigl((\widetilde{T'}\times_T \widetilde{T'})_\red\bigr).
}
\]
Indeed, by smooth descent we can assume that $S$, $T$ and $\widetilde{T'}$ are
schemes so that \cite[Thm.~7.4]{MR2679038} applies.
We conclude that $\tilde{f}$ comes from a unique morphism $h\colon T\to X$.
\end{proof}

We now have our last general reduction result.
\begin{proposition}\label{P:QC-descent:proper}
Let $w\colon W\to Z$ be a proper surjective morphism over $S$. Assume that $Z\to S$ has geometrically normal fibers and $W \to S$ is flat. If $\underline{\Hom}_S(W,g)$ is quasi-compact, then
so is $\underline{\Hom}_S(Z,g)$.
\end{proposition}
\begin{proof}
We may assume that $S$ is an integral scheme. After replacing $S$ with an open
subscheme, we may also assume that $W\to Z$ is flat over an open subset
$U\subseteq Z$ that is dense in every fiber over $S$ and $W\times_Z W$ is flat over $S$. It remains to prove that $H_{Z/S,g}(y) \to T$ is quasi-compact, where $T$ is an integral scheme of finite type over $S$ and $y\colon Z_T \to Y$ is a morphism. By assumption, $H_{W/S,g}(yw) \to T$ is quasi-compact. Now consider the sequence:
\[
\xymatrix{
H_{Z/S,g}(y)\ar[r] & H_{W/S,g}(yw) \ar@<.5ex>@{->}[r]\ar@<-.5ex>@{->}[r] & H_{W\times_Z W/S,g}(yv),
}
\]
where $v\colon W\times_Z W \to Z$ is the natural map. There is a canonical morphism $\varphi\colon H_{Z/S,g}(y)\to E$, where $E$ denotes
the equalizer of the parallel arrows. Since $H_{W/S,g}(yw)$ is quasi-compact (and
$H_{W\times_V W/S,g}(yv)$ is quasi-separated), the equalizer $E$ is
quasi-compact. It is thus enough to show that $\varphi$ is quasi-compact.
Thus, pick a scheme ${T'}$ and a morphism ${T'}\to E$ and let us show that
$H_{Z/S,g}(y)\times_E {T'}$ is quasi-compact.

By noetherian induction on ${T'}$, we may assume that ${T'}$ is normal. The morphism
${T'}\to E$ gives an element of $\Hom_Y(W_{T'},X)$ such that the two images
in $\Hom_Y(W_{T'}\times_{Z_{T'}} W_{T'},X)$ coincide. Noting that 
$Z_{T'}$ is normal, Lemma 
 \ref{L:h-descent} applies to $W_{T'}\to Z_{T'}$ and gives a unique element in $\Hom_Y(Z_{T'}
,X)=\Hom_T({T'},H_{Z/S,g}(y))$. Thus, the morphism $\varphi_{T'}\colon H_{Z/S,g}(y)\times_E {T'}\to {T'}$ has a
section. Repeating the argument with ${T'}=\Spec \kappa(t')$ for every point $t'\in
{T'}$, we see that $\varphi_{T'}$ is injective, so the section is
surjective. It follows that $H_{Z/S,g}(y)\times_E {T'}$ is quasi-compact.
\end{proof}
\subsection*{Proof of the main result}
\begin{proof}[Proof of Theorem~\ref{T:rel-boundedness-affine-fibers}]
As usual, we may assume that $S$ is an affine integral scheme. By
Lemma~\ref{L:QC-descent:normalization}, we may in addition assume that
$Z\to S$ has geometrically normal fibers. Let $W\to Z$ be a proper surjective
morphism with $W$ a projective $S$-scheme~\cite{MR2183251}. By replacing
$S$ with a dense open, we may assume that $W\to S$ is flat. By
Proposition~\ref{P:QC-descent:proper}, we may replace $Z$ with $W$ and assume
that $Z$ is a (projective) scheme. Repeating the first reduction, we may still
assume that $Z\to S$ has geometrically normal fibers.

As before, it remains to prove that $H_{Z/S,g}(y) \to T$ is quasi-compact, where $T$ is an integral $S$-scheme of finite type and $y\colon Z_T \to Y$ is an $S$-morphism. Hence, it suffices to prove the following claim.\\
\emph{Claim:} Let $S$ be integral. If $Z \to S$ is projective with geometrically normal fibers and $q\colon Q \to Z$ is representable with affine fibers, then $\weilr_{Z/S}(Q) \to S$ is quasi-compact.\\
\emph{Proof of Claim:}
Let $\overline{Q}=\Spec_Z(q_*\sO_Q)$ and let $Q\to \overline{Q}\to Z$ be the
induced factorization.  Since $Q\to Z$ has affine fibers, $Q\to \overline{Q}$ is an isomorphism over an open dense subset $U\subseteq Z$. After replacing $S$ with a dense
open subscheme, we may assume that $U$ is dense in every fiber over $S$. Since
$\weilr_{Z/S}(\overline{Q})\to S$ is affine \cite[Thm.~2.3(i)]{hallj_dary_g_hilb_quot}, it is enough to prove that
$\weilr_{Z/S}(Q)\to \weilr_{Z/S}(\overline{Q})$ is quasi-compact. We may thus replace $Q$,
$Z$, $U$ and $S$ with $Q\times_{\overline{Q}} (Z\times_S \weilr_{Z/S}(\overline{Q}))$, $Z\times_S
\weilr_{Z/S}(\overline{Q})$, $U\times_S \weilr_{Z/S}(\overline{Q})$ and $\weilr_{Z/S}(\overline{Q})$.  We may
thus assume that $Q\to Z$ is an isomorphism over $U$.

Since $Q$ is an algebraic space, there exists a finite surjective morphism
$\tilde{Q}\to Q$ such that $\tilde{Q}$ is a scheme. In particular, there is a finite field
extension $L/K(U)$ such that the normalization of $Q$ in $L$ is a scheme.
Take a splitting field $L'/L$ and let $Z'$ be the normalization of $Z$ in
$L'$. Then $Q':=(Q\times_Z Z')_\norm=Q_{\norm/L'}$ is a scheme. By replacing $S$
with a normalization in an extension of $K(S)$ and shrinking, we may assume that
$Z'\to S$ and $Q'\to S$ are flat with geometrically normal fibers.
By Proposition~\ref{P:QC-descent:proper}, it is enough to prove that
$\weilr_{Z'/S}(Q\times_Z Z')$ is quasi-compact.

There is a natural morphism $\weilr_{Z'/S}(Q')\to \weilr_{Z'/S}(Q\times_Z Z')$, which we claim is surjective. To see this, we may assume that $S$
is the spectrum of an algebraically closed field. Then $Z'$ and $Q'$ are normal and any section $Z'\to Q\times_Z Z'$
lifts uniquely to a section $Z'\to Q'$. Indeed, $Z'\times_{Q\times_Z Z'} Q'\to
Z'$ is finite and an isomorphism over $U$, hence has a canonical section. We
can thus replace $Q$ and $Z$ with $Q'$ and $Z'$ and assume that $Q$ is a
scheme.

Since $Q$ is a scheme, it is locally separated; hence, there is
a $U$-admissible blow-up $Z'\to Z$ such that the strict transform $Q'\to Z'$ of $Q\to Z$ is
\'etale~\cite[Thm.~5.7.11]{MR0308104}. After shrinking $S$, we may assume that
$Z'\to S$ is flat. Then since $U\subseteq Z'$ remains dense after arbitrary
pull-back over $S$, we have that $\weilr_{Z'/S}(Q\times_Z
Z')=\weilr_{Z'/S}(Q')$. Replacing $Q\to Z$ with $Q'\to Z'$ 
(Proposition~\ref{P:QC-descent:proper}), we may thus assume that $Q\to Z$ in addition is \'etale.

Finally, we note that the \'etale morphism $Q\to Z$ corresponds to a
constructible sheaf on $Z_{\mathrm{\acute{E}t}}$ and
that $\weilr_{Z/S}(Q)$ is nothing but the \'etale sheaf $f_{\mathrm{\acute{E}t},*}Q$. By a special 
case of the proper base change theorem \cite[XIV.1.1]{MR0354654}, 
$f_{\mathrm{\acute{E}t},*}Q$ is constructible, so $\weilr_{Z/S}(X)\to S$ is of
finite presentation.

For the second part of the theorem on groups: let $T$ be the the spectrum of an algebraically 
closed field and let $y\colon Z_T \to Y$ be a morphism. By the first part $H_{Z/S,g}(y)\simeq\weilr_{Z_T/T}(Q)$ is then a
group scheme $G$ of finite type over $T$, where $Q=X\times_Y Z_T$. Let $K=G_{\ant}$ be the largest
anti-affine subgroup of $G$; it is normal, connected and smooth and the
quotient $G/K$ is affine~\cite[\S III.3.8]{MR0302656}.

The universal family $G\times_T Z_T\to Q$ is a group
homomorphism and induces a group homomorphism $K\times_T Z_T\to Q$. It is enough to show that this factors through the unit section of $Q\to Z_T$, because this forces $K=0$.

Note that for every stack $W\to T$, the pull-back $H\times_T W\to W$ is
an anti-affine group in the sense that the push-forward of $\sO_{H\times_T W}$
is $\sO_W$ (flat base change). Since $Q\to Z_T$ has affine fibers, there is a
finitely presented filtration $(Z_i)$ of $Z_T$ with strata $V_i^{[n]}$ over which
$Q\times_{Z_T} V_i^{[n]}\to V_i^{[n]}$ is affine. Since $K\times_T V_i^{[n]}$ is
anti-affine, it follows that $K\times_T V_i^{[n]}\to
Q\times_{Z_T} V_i^{[n]}$ factors through the unit section $V_i^{[n]}\to Q\times_{Z_T}
V_i^{[n]}$.

Let $E$ be the equalizer of $K\times_T Z_T \to Q$ and the constant map $K\times_T Z_T \to Q$ to the unit. The above discussion shows that the monomorphism $E\to K\times_T
Z_T$ is an isomorphism over every strata $V_i^{[n]}$, hence an isomorphism
(Proposition~\ref{P:mono-with-stratification-is-iso}, using that the filtration
is separating since $Z_T$ is noetherian).
\end{proof}
\bibliography{bibtex_db/references}
\bibliographystyle{bibtex_db/dary}

\end{document}